\documentclass[reqno, 12pt]{amsart}
\usepackage{a4wide, amsfonts, amssymb, amsmath, amsthm, mathrsfs, enumerate, enumitem, setspace, url}
\usepackage[utf8]{inputenc}

\usepackage{graphicx}
\usepackage[dvipsnames]{xcolor}
\usepackage{tikz}
\usetikzlibrary{positioning}
\usetikzlibrary{intersections}
\usetikzlibrary{calc}
\usetikzlibrary{decorations.markings}
\usetikzlibrary {shapes.geometric}
\usepackage{tikz-cd} %diagrams
	\usetikzlibrary{cd}
	\usetikzlibrary{graphs}
\usepackage[all]{xy}
\usepackage{mathtools}

\usepackage{color}
\usepackage{algpseudocode}
\usepackage{hyperref} %hyperlinks
	\definecolor{darkred}{rgb}{0.5,0,0}
	\definecolor{darkgreen}{rgb}{0,0.5,0}
	\definecolor{darkblue}{rgb}{0,0,0.5}
	\hypersetup{linkcolor=darkblue,filecolor=darkgreen,urlcolor=darkred,citecolor=darkblue}

\usepackage{comment}
\usepackage{longtable}
\usepackage[T2A,T1]{fontenc}
\DeclareSymbolFont{cyrillic}{T2A}{cmr}{m}{n}
\DeclareMathSymbol{\Sha}{\mathalpha}{cyrillic}{216}

\usepackage[toc,page]{appendix}

\allowdisplaybreaks
\theoremstyle{plain}
\newtheorem{theorem}{Theorem}[section]
\newtheorem*{theorem*}{Theorem}
\newtheorem{proposition}[theorem]{Proposition}
\newtheorem{lemma}[theorem]{Lemma}
\newtheorem{corollary}[theorem]{Corollary}

\theoremstyle{remark}
\newtheorem{remark}[theorem]{Remark}
\newtheorem{example}[theorem]{Example}

\newtheorem*{acknowledgements}{Acknowledgements}

\theoremstyle{definition}

\newtheorem{definition}[theorem]{Definition}

\numberwithin{equation}{section}
    \DeclareFontFamily{U}{wncy}{}
    \DeclareFontShape{U}{wncy}{m}{n}{<->wncyr10}{}
    \DeclareSymbolFont{mcy}{U}{wncy}{m}{n}
    \DeclareMathSymbol{\Sh}{\mathord}{mcy}{"58} 
\DeclareMathOperator{\Image}{Im}
\DeclareMathOperator{\Real}{Re}
\newcommand{\NN}{\mathbb{N}}
\newcommand{\Z}{\mathbb{Z}}

\newcommand{\ZZ}{\mathbb{Z}}
\newcommand{\QQ}{\mathbb{Q}}

\newcommand{\RR}{\mathbb{R}}
\newcommand{\CC}{\mathbb{C}}

\newcommand{\AS}{\ensuremath{\mathbf{A}}}% adeles

\renewcommand{\ge}{\geqslant}
\renewcommand{\le}{\leqslant}
\newcommand{\field}[1]{\mathbb{#1}}
\newcommand{\gal}{\mathrm{Gal}}

\newcommand{\F}{\field{F}}

\newcommand{\kk}{\mathit{k}_v^*}

\newcommand{\OO}{\mathcal{O}_v^*}
\newcommand{\OOO}{\mathcal{O}_v}
\newcommand{\OOone}{\mathcal{O}_{v,1}^*}
\newcommand{\OOtwo}{\mathcal{O}_{v,2}^*}

\newcommand{\OOjs}{\mathcal{O}_{v,j-1}^*}
\newcommand{\OOj}{\mathcal{O}_{v,j}^*}
\newcommand{\OOjj}{\mathcal{O}_{v,j+1}^*}

\newcommand{\Frob}{\operatorname{Frob}}

\usepackage{faktor}
\usepackage[all]{xy}

\DeclareMathOperator{\Gal}{Gal}

\DeclareMathOperator{\Hom}{Hom}

\renewcommand{\epsilon}{\varepsilon}

\begin{document}
\onehalfspacing
\title{Equidistribution for abelian extensions of global fields}

\author[Jiazhi He]{ 
Jiazhi He}
	\address{Jiazhi He, Department of Mathematics, University of Bath, Bath, BA2 7AL, UK.}
	\email{jh3671@bath.ac.uk}

%\date{\today}
%\newcommand{\creflastconjunction}{, and\nobreakspace}
%\thanks{2020 {\em Mathematics Subject Classification} 
%	 ??.
%}
\begin{abstract}
We establish asymptotic formulas for abelian extensions of global function fields ordered by conductor and subject to prescribed local conditions.  Our proof combines harmonic analysis with a theory of frobenian functions over global function fields developed in this paper. We interpret our result via equidistribution on algebraic stacks. 
\end{abstract}
\maketitle
\setcounter{tocdepth}{1}
\tableofcontents
\section{Introduction}
In \cite{Malle2002} and \cite{Malle2004}, Malle proposed a conjecture on the asymptotics for the counting function for extensions over number fields with fixed Galois group and bounded norm of the discriminant. Later, Bhargava, in \cite[Sec.~8.2]{10.1093/imrn/rnm052}, proposed a heuristic for refining this counting problem,  where he counts global fields imposing local specifications and calculates the quotient of this count to the total count. According to the terminology from \cite[\S 9]{loughran2025mallesconjecturebrauergroups}, where the authors developed a geometric framework for counting field extensions, such questions are referred to as \textit{equidistribution}. Wood studied equidistribution over abelian number fields in  \cite{MR2581243}. 

In this paper, we study equidistribution for abelian extensions over global function fields with characteristic $p$. For a non-trivial finite abelian group $G$, we choose a presentation  $G = G_w\times G_t$, where $G_w = \prod_{m=1}^e C_{p^m}^{r_m}$ is the $p$-primary part of $G$. We say that $G$ is \textit{tame} if the wild part $G_w$ is trivial and that $G$ is \textit{wild} if the tame part $G_t$ is trivial.  Let $0\leq f \leq e$ denote the minimal integer such that $p^fG_w$ is cyclic. 

Let $F$ be a field with separable closure $\Bar{F}$. We define a \textit{sub-$G$-extension} of $F$ as a continuous homomorphism $\Gamma_F:=\gal (\Bar{F}/F) \to G.$ A  \emph{$G$-extension} is a surjective  sub-$G$-extension, and we write $G\text{-}\operatorname{ext}(F)$ for the set of $G$-extensions of $F$. After fixing compatible embeddings, a sub-$G$-extension $\varphi$ of $k$ will induce a sub-$G$-extension $\varphi_v$ of $k_v$, where $k_v$ is a completion of $k$ at the place $v$. Let $W \subseteq \prod_v\Hom(\Gamma_{k_v},G)$.  By abuse of notation we write $\varphi \in W$ if $ \Image (\varphi) \in W$ under the map $ \Hom(\Gamma_{k},G) \to \prod_v\Hom(\Gamma_{k_v},G)$.   Let $\Phi(\varphi)$ denote the norm of the conductor of the $G$-extension $\varphi \in \operatorname{Hom}(\Gamma_k, G)$. We are interested in the counting function
\begin{equation}
    N(k,G,W,q^M) := {\#\left\{\varphi \in  G-\operatorname{ext}(k) : 
    \varphi \in W,  
     \Phi(\varphi) = q^{M}
    \right\}}. 
\end{equation}
The function field setting exhibits several new phenomena absent in number fields.
In particular, the presence of constant field extensions introduces additional poles in the associated Dirichlet series, and the wild part of the Galois group interacts subtly with the analytic behavior of the generating functions. Therefore, we will state our results using an averaged version of $ N(k,G,W,q^M)$:
\begin{equation}
    N_{av}(k,G,W,q^M) := \frac{1}{|B(k,G)|}\sum_{j = 0}^{|B(k,G)|-1} \frac{N(k,G,W,q^{M+j})}{q^{a(G) j}},
\end{equation}
where $a(G) =p^{-e}(1+ \sum_{m=1}^er_m(p^e- p^{e-m}))$, and
\begin{equation}\label{Eq:formulaforB}
    |B(k,G)| = 
    \begin{cases}
        \gcd\Bigl(\{j : 2\leq j \leq p^e\}\cup\bigl\{[k(\mu_\ell):k] : \ell \text{ prime},\ \ell \mid |G_t|\bigr\}\Bigr), &\text{ if $G_w$ is cyclic,}\\
        p^f, &\text{ if $G_w$ is not cyclic},
    \end{cases}
\end{equation}
where $k(\mu_\ell)/k$ denotes the constant extension obtained by adjoining the $\ell$-th roots of unity to $k$. In the first case, the set $\{j : 2\leq j \leq p^e\}$ is empty when $G_w$ is trivial, the second set is empty when $G_t$ is trivial, and the union is non-empty as $G$ is non-trivial.

In particular, $|B(k,G)|=1$ when $G_w$ is cyclic, non-trivial and $p^e\geq 3$; and $|B(k,G)|=\gcd\{[k(\mu_\ell):k] : \ell\mid|G_t|\}$ when $G$ is tame.

We use $N(k,G,q^M)$ and $ N_{av}(k,G,q^M) $ respectively for the case $W = \prod_v\Hom(\Gamma_{k_v},G)$.
Note that $N_{av}(k, G, W, q^M)$  depends on the number $|B(k, G)|$ which may appear ad hoc; this is a new feature over global function fields. See Remark~\ref{Remark: periodofpole} for details.

Now we are in a position to state our first version of the equidistribution theorem.
\begin{theorem}\label{easyED}
Let $k$ be a global function field, $G$ be a finite abelian group and let $S$ be a finite set of places of $k$. For $ \psi_v \in \Hom (\Gamma_{k_v},G)$ with $v \in S$, let $W = \prod_{v \in S}\{\psi_v\} \times \prod_{v \notin S}\Hom(\Gamma_{k_v},G)$. Let $a(G):=p^{-e}(1+ \sum_{m=1}^er_m(p^e- p^{e-m}))$. Then
\begin{equation*}
    \lim_{M \to \infty}\frac{N_{av}(k,G,W,q^M)}{N_{av}(k,G,q^M)} =\prod_{v \in S} \frac{\frac{1}{\Phi_v(\psi_v)^{a(G)}}}{\left( \sum_{\chi_v \in \operatorname{Hom}(k_v^*,G)} \frac{1}{\Phi(\chi_v)^{a(G)}}\right)}.
\end{equation*}
\end{theorem}
\begin{remark}
Given the positivity of the above quotient, we have that, for a collection of local $G$-extensions $W = \prod_{v \in S}\{\psi_v\} \times \prod_{v \notin S}\Hom(\Gamma_{k_v},G)$, there exists a global $G$-extension $\varphi$, such that $\varphi_v = \psi_v$ for all $v \in S$. This result can also be found in \cite[Thm.~9.2.8]{CohoNF}.
    
    This stands in sharp contrast to the number field case, where the Grunwald--Wang theorem asserts that there is no $\mathbb{Z}/8\mathbb{Z}$-extension of $\mathbb{Q}$ realizing the degree-8 unramified extension of $\mathbb{Q}_2$; see page 1 of \cite{Wang1948} for details. 
    If $G$ is purely wild, we have a new proof of this fact by counting. If $G$ is tame, our result requires a fact equivalent to Grunwald--Wang, namely, if $x \in k^*$ is a $n$-th power almost everywhere locally, then $x$ is a $n$-th power itself.
\end{remark}
\begin{theorem}\label{maintheorem} Let $k$ be a global function field with characteristic  $p$ and let $G$ be a finite abelian group. Write $G = G_w\times G_t$, where $G_w = \prod_{m=1}^e C_{p^m}^{r_m}$ is the $p$-primary part of $G$ and $e$ is the exponent of $G_w$. If $G_w$ is not cyclic, let $0\leq f \leq e$ denote the minimal integer such that $p^fG_w$ is cyclic.  For $ \psi_v \in \Hom (\Gamma_{k_v},G)$ with $v \in S$, let $W = \prod_{v \in S}\{\psi_v\} \times \prod_{v \notin S}\Hom(\Gamma_{k_v},G)$. For any positive integer $ \alpha \geq 1$ , let $k(\mu_\alpha)/k$ be the constant extension by adjoining $\alpha$-th roots of unity 
    to $k$, define $d(k,\alpha) = \frac{r(\alpha)\phi(\alpha)}{[k(\mu_\alpha):k]}$, where $r(\alpha)$ is the number of cyclic subgroups with order $\alpha$ in $G_t$.
    Let
     \begin{equation}\label{bG}
        b(k,G) = \begin{cases}
        p^e\left( \sum_{\alpha\,\mid \, |G_t|}d(k,\alpha)\right) - 1 , &\text{if $G_w$ is cyclic,}
            \\
        \left( \sum_{\alpha\,\mid \, |G_t|}d(k,\alpha)\right)p^{e-f}, &\text{otherwise.}\end{cases}
    \end{equation}
Then we have  
$$
N_{av}(k,G,W,q^M)\sim q^{a(G)M}c(k,G,W) M^{b(k,G)-1} , \quad \text { as } M \rightarrow \infty \text { and } M \in\Z, \text{ where}$$
\begin{align*}
    c(k,G,W) &= \frac{\log (q)^{b(k,G)}}{(b(k,G)-1)!}\frac{ |G|}{|\F_q^* \otimes G|}\left(\prod_{\substack{1\leq j \leq p^e \\ p^f\mid j }}\frac{1}{j}\right)^{\left( \sum_{\alpha\,\mid \, |G_t|}d(k,\alpha)\right)}
    \\ &\times\left(\operatorname{Res}_{s = 1}\mathrm{L}(M_G,s)\right)\prod_{v \in S}\frac{\lambda_v^{-1}}{|G|\Phi_v(\psi_v)^a}\prod_{v \notin S}\left( \sum_{\chi_v \in \operatorname{Hom}(k_v^*,G)} \frac{1}{\Phi(\chi_v)^a}\right)\lambda_v^{-1}.
\end{align*}
Here, the infinite product converges absolutely and   
$\mathrm{L}(M_G, s)$ denotes the Artin $L$-function with corresponding local Euler factors $\mathrm{L}_v(M_G, s)$ where

\begin{equation}\label{GaloisModule}
    M_G= \begin{cases}
\left(\CC^{d(k,\alpha)}\right)^{p^e-1}\oplus\CC^{d(k,\alpha)-1}\oplus \left(\oplus_{\substack{ \alpha\,\mid \, |G_t|\\\alpha \neq 1}}\operatorname{Ind}_{k}^{k(\mu_\alpha)}(\CC) 
 ^{d(k,\alpha)}\right)^{p^{e}}, &\text{  $G_w$ cyclic,}
 \\\left(\oplus_{\alpha\,\mid \, |G_t|}\operatorname{Ind}_{k}^{k(\mu_\alpha)}(\CC) 
 ^{d(k,\alpha)}\right)^{p^{e-f}} , &\text{otherwise,}
\end{cases}
\end{equation}
and
$
\lambda_v=\mathrm{L}_v(M_G, 1).
$
\end{theorem}
We have explicitly calculated the leading constant in Theorem \ref{maintheorem}. There is a significant interest in the leading constant of Malle's conjecture, due to the recent paper \cite{loughran2025mallesconjecturebrauergroups} of Loughran and Santens, where they put forward the conjecture over number fields.  Since their construction does not work over function fields with groups that have non-trivial wild parts,  one of the aims of the paper is to help generate more examples, particularly for the wild case over a function field. In \S 6, we explain how to interpret the leading constant using algebraic stacks.  We also define a Tamagawa measure $\tau_\Phi$ in that section. Using this, together with the study of the topology, we prove a stronger version of Theorem \ref{easyED} over global function fields. Recall, a \textit{continuity set} is a set whose boundary has measure zero.
\begin{theorem}\label{ED}
Let $G$ be a finite abelian group and $W \subseteq \prod_{v}BG[k_v]$ be a continuity set. Then
\begin{equation*}
    \lim_{M \to \infty}\frac{N_{av}(k,G,W,q^M)}{N_{av}(k,G,q^M)} = \frac{\tau_\Phi(W)}{\tau_\Phi(\prod_vBG[k_v])}.
\end{equation*}
\end{theorem} 
\subsection{Outline of the proof} Let $W = \prod_{v \in S}\{\psi_v\} \times \prod_{v \notin S}\Hom(\Gamma_{k_v},G)$, and let $f_W$ denotes the indicator function of $\varphi \in \Hom(\Gamma_k,G)$ for which $\varphi \in W.$ Hence, $f_{W}(\varphi) = \prod_v f_{W_v}(\varphi_v)$  where  $f_{W_v}$ is the  local indicator function of $ \varphi_v = \psi_v$ for $v \in S$ and $f_{W_v} =1$ for $v \notin S$. We define
\begin{equation}\label{DserieswithLC}
    F_{G,W}(s)=\sum_{\varphi \in G\text{-}\operatorname{ext}(k)} \frac{f_{W}(\varphi)}{\Phi(\varphi)^s},
\end{equation}
so that $F_{G,W}(s)=\sum_{M\geq 0}N(k,G,W,q^M)\,q^{-Ms}$. Note that this is only a formal sum at this point.

$F_{G,W}(s)$ is periodic, as it is a power series in $q^{-s}$. Therefore, we work in the following strip:
 \begin{equation}\label{fundamentalRegionD(shift this)}
 D= \left\{ s\in\CC: -\frac{\pi i}{\log q} \le \Image(s) <  \frac{\pi i}{\log q}\right\}.
\end{equation}
Theorem \ref{maintheorem} is based on the analytic behaviour of the generating series $F_{G, W}(s)$, we show that there exists $\eta >0$ such that  $F_{G,W}(s)$ defines a holomorphic function for $\Real(s) \geq a(G) - \eta$ except on the line $\Real(s) = a(G),$ where there are finitely many  poles in $D$ with order at most $b(k, G)$. 

To prove Theorem \ref{maintheorem}, we analyze the corresponding Dirichlet series (\ref{DserieswithLC}). We aim to study this series using the harmonic analysis method developed in \cite{MR3884640}. This is the first time that this method has been applied to global function fields rather than number fields. A new difficulty arises in this setting because every place can be wild; consequently, even the convergence of the local Fourier transforms is not obvious, as it is an infinite sum.

In order to analyse the associated Dirichlet series, we also develop a theory of Frobenian functions over global function fields. This notion goes back to Serre in the number field setting \cite[\S3.3]{serre2011lectures}. This framework has found numerous applications in counting problems over number fields, see \cite{Gundlach_2026}, \cite{Frei_2022}, \cite{loughran2022frobenianmultiplicativefunctionsrational}, \cite[\S6]{Alberts_2026}. In the function field setting, however, the presence of the constant field leads to additional phenomena; the Artin conjecture is only valid for geometric extensions, and for non-geometric extensions a  modification is required. We address this difficulty and adapt the Frobenian framework to global function fields. 
\subsection{Existing results} $N(k,G,q^M)$ is in the realm of Malle's conjecture. For the case when $G$ is a finite abelian group, Wright \cite{10.1112/plms/s3-58.1.17} obtained an asymptotic of the counting function for the discriminant over number fields and over global function fields with $G$ being tame. Since Wright's work, there have been numerous works on removing this assumption from Wright's paper for the function field case.

A considerable amount of work has already been devoted to counting extensions over $\mathbb{F}_q(T)$. In the abelian case, \cite{gundlach2025countingabelianextensionsartinschreier} studies abelian $p$-group extensions over $\mathbb{F}_q(T)$ when $G$ is  wild. Other authors have considered counting problems with respect to different height functions: for example, \cite{potthast2024asymptoticselementaryabelianextensionslocal} works with the discriminant, while \cite{gundlach2025countingtwostepnilpotentwildly} and \cite{gundlach2026liftsunramifiedtwistslocalglobal} use the last jump in the ramification filtration when $G$ is wild. For non-abelian and tame $G$, Santens \cite{santens2026leadingconstantmallesconjecture} proves an asymptotic formula for general heights providing $q$ is sufficiently large.  These results make fundamental use of the fact that the ground field is $\F_q(T)$; for example, $\F_q(T)$ has trivial class group.

For general global function fields, Lagemann studied the case where $G$ is  wild. He obtained the asymptotic formula for  wild cyclic groups in \cite{Lagemann_2012} and for arbitrary wild groups in \cite{Lagemann_2015} when counting by the conductor. We allow general function fields over any finite abelian group $G$, and our result also allows us to impose finitely many local conditions. 
\begin{remark}\label{Remark: periodofpole}
Due to the periodic nature of the generating series of the corresponding counting function, we work in the region $D$ from \eqref{fundamentalRegionD(shift this)}.

$|B(k, G)|$, as defined in \eqref{Eq:formulaforB} is actually the number of rightmost poles of the corresponding generating series with highest order in the region $D$. Since there may be multiple poles of highest order in the region $D$, it is necessary to account for possible cancellations among them. 
To deal with possible cancellations, Lagemann restricted his results to some arithmetic progressions in \cite[Thm.~1.2]{Lagemann_2012} and \cite[Thm.~1.2]{Lagemann_2015}. 
We instead introduce a weighted average $N_{av}(k,G,W,q^M)$ to handle potential cancellations. 
As we shall see in \S 4, this approach allows us to express the leading constant more neatly, whereas Lagemann did not compute the leading constant explicitly in \cite[\S 1]{Lagemann_2015}.
\end{remark}
\subsection{Notations and Conventions}
We fix a function field $k$ with characteristic $p$, whose Dedekind zeta function is denoted by $\zeta_k(s)$. We use $\mathcal{O}_k$ to denote the ring of integers of $k$ and use  $\mathcal{O}_S$ for the $S$-integers of $k$ for $S \subseteq \Omega_k$ where $\Omega_k$ is the set of all places of $k$. 

For $v \in \Omega_k$, we let  $k_v$ be the completion of $k$ at $v$, and let $\mathcal{O}_v$ be the ring of integers of $k_v$ and  $\mathbb{F}_v$ be the residue field at a finite place $v$. We denote the cardinality of the residue field at a finite place $v$ by $q_v$. We denote by $\AS^{*}_k$ the ideles of $k$.  
 
In this paper, all finite groups are viewed as topological groups with the discrete topology, and for topological groups $A$ and $B$, denote by $\Hom(A, B)$ the continuous homomorphisms from $A$ to $B$ with open-compact topology.
\begin{acknowledgements}
I would like to thank Daniel Loughran for suggesting the problem and for his many suggestions that greatly improved this paper. I am grateful to Will Sawin for his help with the proof of Theorem~2.2, Fabian Gundlach for useful comments, and Abdulmuhsin Alfaraj for useful discussions. This research was supported by EPSRC.

\end{acknowledgements}

\section{Frobenian functions over global function fields}
In \cite[\S~2]{Frei_2022}, the authors study the problem of counting abelian extensions over number fields.
To investigate the analytic behavior of the associated Dirichlet series, they introduce frobenian functions and analyze the resulting Dirichlet series using the Artin L-function.

We aim to adapt this strategy to the function field setting.
Let $L/k$ be a finite Galois extension of global function fields.
Denote by $K$ the maximal constant sub-extension of $L$. Then $L/K$ is geometric, in the sense that it has the same constant field as $K$,
while $K/k$ is a constant extension.
Accordingly, one obtains a canonical short exact sequence
\begin{equation}\label{Seq:Const-Geo_Decomp} 
1 \longrightarrow \mathrm{Gal}(L/K) \longrightarrow \mathrm{Gal}(L/k) \longrightarrow \mathrm{Gal}(K/k) \longrightarrow 1.
\end{equation}
Thus, every finite extension of global function fields admits a natural decomposition into its geometric and constant parts.
The main difficulty is that the usual holomorphy statement for Artin $\operatorname{L}$-functions only applies to the geometric part, while the constant part contributes additional poles.
Consequently, the argument of \cite[\S~2]{Frei_2022} does not directly extend to the function field setting,
and a different treatment is needed for the contribution of the constant extension.

\subsection{Artin L-functions
}
Let $k$ be a function field of genus $g$ with constant field $\F_q$. According to \cite[Thms.~5.9 and 5.10]{Rosen_2002}, we have the basic properties of the zeta function of $k$:

\begin{lemma}\label{lemma:value_zeta_function}
We have
\begin{equation*}
\zeta_{k}(s)=\frac{\mathcal L(q^{-s})}{(1-q^{1-s})(1-q^{-s})},\qquad \Real(s)>1,
\end{equation*}
where $\mathcal L(u)\in\Z[u]$ has degree $2g$, $\mathcal L(0)=1$ and $\mathcal L(1)>0$. The right-hand side gives a meromorphic continuation of $\zeta_k(s)$ to $\Real(s)>0$, with simple poles at $s=1+\frac{2\pi i n}{\log q}$ for $n\in\Z$ and no other poles there. All roots of $\mathcal L(u)$ have absolute value $q^{-1/2}$; in particular, $\zeta_k(s)\neq 0$ for $\Real(s)>\frac{1}{2}$.
\end{lemma}

Recall that if $K/k$ is a constant extension, we have that 
\begin{equation}\label{eq:decomp_of_DEDEKIND_CONST_EXT}
    \zeta_K(s)= \zeta_k(s)\prod_{\substack{\chi:\operatorname{Gal}(K/k)\to \CC^* \\ \chi \text{ irreducible, }  \chi \neq1}} \operatorname{L}(\chi,s),
\end{equation}
where $\operatorname{L}(\chi,s)$ is the Artin $\operatorname{L}$-function associated to the representation $\chi$. As we are working over global function fields, sometimes it is easier to use the change of variable $u = q^{-s}$ with the notation defined in Rosen's book: $Z(\chi,u)=\operatorname{L}(\chi,s)$ and $ Z_k(u) =  \zeta_k(s).$

A representation  $\chi$ is called
a \emph{constant character} if it is trivial on $\Gal(L/K)$, equivalently, $ \chi$ factors through $\Gal(K/k)$ as in \eqref{Seq:Const-Geo_Decomp}. A constant character is one-dimensional if it is irreducible and non-trivial.
In which case, $\chi(\Frob_q)$ is well defined and given by $\chi (\sigma)$, where $\sigma$ is any lift of $\Frob_q \in \gal(K/k)$ to $\gal(L/k)$, since $\chi$ is a group homomorphism. Moreover,
\begin{equation}\label{eq:chi-frobv}
\chi(\Frob_v)=\chi(\Frob_q)^{\deg(v)}.
\end{equation}
The next theorem can be regarded as a version of Artin's conjecture for global function fields.  Many standard formulations only address the geometric case; by contrast, the statement below covers arbitrary finite Galois extensions.

\begin{theorem}\label{lemma_Sawin}
Let $\chi$ be an irreducible representation of $\Gal(L/k)$.
Then $\operatorname{L}(\chi,s)$ is holomorphic if and only if $\chi$ is not constant; in this case, $Z(\chi,u)$ is a polynomial whose roots have absolute value $q^{-1/2}$. Moreover, if $ \chi$ is constant, then
\begin{equation}\label{eq_Z(chiu)_is_shift}
    Z(\chi,u) = Z_k(\chi(\operatorname{Frob}_q)u).
\end{equation}
\end{theorem}
\begin{proof}
For the holomorphicity of $\operatorname{L}(\chi,s)$ for characters $\chi$ that do not factor through $\Gal(K/k)$, 
a statement appears in \cite[Lemma~1]{MurtyScherk1994}. 
However, the final step of the proof given there relies on the implicit assumption that 
if a product of rational functions is a polynomial, 
then each factor is itself a polynomial, 
which does not hold in general. 
A complete proof can instead be found in \cite{MO_Sawin_2026_ArtinFF}. The Riemann hypothesis then follows as in \cite[Lemma~1]{MurtyScherk1994}.

Let $\chi$ be a constant character; the local Euler factor of $Z(\chi,u)$ at an unramified place $v$ is
\[
\operatorname{L}_v(\chi,u)
=
\det\left(
1-\chi(\Frob_v)\,u^{\deg v}
\right)^{-1}.
\]
As $\chi$ is constant, the determinant reduces to a scalar. At a place $v$ that ramifies in $L/k$, the inertia subgroup $I_v$ has trivial image in $\Gal(K/k)$ by \eqref{Seq:Const-Geo_Decomp}, hence $I_v\subseteq\ker\chi$ and the local factor takes the same form.
Taking the product over all places $v$ of $k$ yields
\[
Z(\chi,u)
=
\prod_v
\left(
1-(\chi(\Frob_q)u)^{\deg v}
\right)^{-1}
=
Z_k\bigl(\chi(\Frob_q)u\bigr),
\]
which is the zeta function of $k$ with the variable $u$ multiplied
by the scalar $\chi(\Frob_q)$.
\end{proof}

\subsection{Frobenian functions}
Our proof of the main result relies on the theory of frobenian functions; we refer to \cite[\S 3.3]{serre2011lectures} for the theory over number fields. 
Recall that a \emph{class function} on a group is a function which is constant on its conjugacy classes.

\begin{definition}\label{def:frobenian}
Let $k$ be a global field and let $\rho : \Omega_k \to \CC$ be a function
on the set of places of $k$. Let $S$ be a finite set of places of $k$.
We say that $\rho$ is \emph{$S$-frobenian} if there exist
\begin{enumerate}
    \item a finite Galois extension $K/k$ with Galois group $\Gamma$ such
    that $S$ contains all places which ramify in $K/k$, and
    \item a class function $\varphi : \Gamma \to \CC$,
\end{enumerate}
such that for all $v \notin S$ one has
\[
\rho(v) = \varphi(\Frob_v),
\]
where $\Frob_v \in \Gamma$ denotes a Frobenius element at $v$.
We say that $\rho$ is \emph{frobenian} if it is $S$-frobenian for some $S$.
A subset of $\Omega_k$ is called $(S\text{-})\emph{frobenian}$ if its
indicator function is $(S\text{-})$frobenian.
\end{definition}

We adopt the usual abuse of notation (see \cite[\S 3.2.1]{serre2011lectures}) and write
$\Frob_v \in \Gamma$ for a choice of an element in the Frobenius conjugacy
class at $v$. This is well-defined as $\varphi$ is a class function.

We define the mean of $\rho$ to be
\[
m(\rho) = \frac{1}{|\Gamma|} \sum_{\gamma \in \Gamma} \varphi(\gamma) \in \CC.
\]

\begin{example}\label{ex:frobenian-polynomial}
Let $f(x) \in k[x]$ be a (not necessarily irreducible) separable polynomial. Then the set
\[
\{ v \in \Omega_k : f(x) \text{ has a root in } k_v \}
\]
is frobenian. Indeed, let $K$ be the splitting field of $f$. For a place $v$
which is unramified in $K$, the polynomial $f$ has a root in $k_v$ if and only
if $\Frob_v$ acts with a fixed point on the roots of $f$ in $\bar{k}$. The set
of such elements is a conjugacy-invariant subset of $\Gamma$, which proves the claim.
\end{example}
\begin{example}
Let $\chi$ be an irreducible constant character of $\Gal(L/k)$ and view $\chi$ as a character of $\Gal(K/k)$. The function $v\mapsto\chi(\Frob_v)=\chi(\Frob_q)^{\deg v}$, cf.\ \eqref{eq:chi-frobv}, is $\emptyset$-frobenian, the constant extension $K/k$ being unramified at every place.
\end{example}

We will require the following auxiliary result concerning non-negative frobenian functions. It is inspired by its number field analogue in \cite[Lemma~2.3]{loughran2022frobenianmultiplicativefunctionsrational}. The proof given there relies on an assertion concerning sums of absolute values of complex numbers that does not hold in general. Replacing this assertion with the correct version yields the following function field analogue.

\begin{lemma}\label{lem_mean_frob_real_nonneg}
Let $k$ be a global function field, let $S$ be a finite set of places of $k$, and let $\rho$ be an $S$-frobenian function associated to $\Gamma=\Gal(L/k)$ which is real-valued and non-negative, and let $\chi$ be an irreducible constant character of $\Gamma$.
\begin{enumerate}
    \item We have $\Real m(\rho\chi) \leq |m(\rho\chi)| \leq m(\rho)$.
    \item The following are equivalent:
    \begin{enumerate}
        \item[(a)] $\Real m(\rho\chi) = m(\rho)$;
        \item[(b)] $m(\rho\chi) = m(\rho)$;
        \item[(c)] for all $n \geq 1$, $\rho\chi^{n}(v) = \rho(v)$ for all but finitely many places $v \in \Omega_k$.
    \end{enumerate}
\end{enumerate}
\end{lemma}

\begin{proof}
Note that $\varphi$ is also real and non-negative.
We thus have
\begin{equation}\label{abs_value_mean_rhochi}
\Real(m(\rho\chi)) \leq |m(\rho\chi)| = \frac{1}{|\Gamma|}\Bigl|\sum_{\gamma\in\Gamma}\varphi(\gamma)\chi(\gamma)\Bigr| \leq \frac{1}{|\Gamma|}\sum_{\gamma\in\Gamma}\varphi(\gamma) = m(\rho),
\end{equation}
as required, on using $|\chi| = 1$.

To prove \textit{(2)}, we use the following fact: if $z_1,\dots,z_m\in\CC$ and
\begin{equation}\label{complex_fact}
|z_1| + \dots + |z_m| = \Real(z_1 + \dots + z_m),
\end{equation}
then $z_i = |z_i|$ for all $i$. Indeed, $\Real(z_i)\leq|z_i|$ for each $i$, with equality if and only if $z_i\in\RR_{\geq 0}$, and \eqref{complex_fact} forces equality in every term.

Assume \textit{(a)} holds. Then equality holds throughout \eqref{abs_value_mean_rhochi}, so by \eqref{complex_fact} with $z_\gamma = \varphi(\gamma)\chi(\gamma)$ we have
\begin{equation}\label{eq:pointwise_stabiliser}
\varphi(\gamma)\chi(\gamma) = |\varphi(\gamma)\chi(\gamma)| = \varphi(\gamma)\qquad\forall\gamma\in\Gamma,
\end{equation}
and averaging over $\Gamma$ gives $m(\rho\chi)=m(\rho)$, whence \textit{(b)}.

Assume \textit{(b)}, so that $\sum_{\gamma\in\Gamma}\varphi(\gamma) = \sum_{\gamma\in\Gamma}\varphi(\gamma)\chi(\gamma).$
Taking real parts and applying \eqref{complex_fact} with $z_\gamma = \varphi(\gamma)\chi(\gamma)$ again gives $\varphi(\gamma)\chi(\gamma)=\varphi(\gamma)$ for all $\gamma\in\Gamma$; for every $\gamma$ with $\varphi(\gamma)\neq 0$ this forces $\chi(\gamma)=1$, so
\begin{equation*}
\varphi(\gamma)\chi^n(\gamma) = \varphi(\gamma)\qquad\forall n\geq 1,\ \gamma\in\Gamma,
\end{equation*}
which proves \textit{(c)}, as $\rho\chi^n(v)=\varphi\chi^n(\Frob_v)=\varphi(\Frob_v)=\rho(v)$ for all places $v\notin S$.

Finally, assume \textit{(c)} and take $n=1$: the class functions $\varphi\chi$ and $\varphi$ agree on the Frobenius classes of all places outside a finite set, and by the Chebotarev density theorem these exhaust the conjugacy classes of $\Gamma$, so $\varphi\chi=\varphi$ and $m(\rho\chi)=m(\rho)$; this proves \textit{(b)}, which  implies \textit{(a)}, as required.
\end{proof}
The proof shows that each condition in Lemma~\ref{lem_mean_frob_real_nonneg}\textit{(2)} is equivalent to the pointwise identity $\varphi\chi=\varphi$ on $\Gamma$. This has the following consequence.

\begin{corollary}\label{cor:extremal_subgroup}
Let $\rho$, $\varphi$ and $\Gamma$ be as in Lemma~\ref{lem_mean_frob_real_nonneg}, let $K$ be the maximal constant sub-extension of $L/k$, and let $X$ denote the group of irreducible constant characters of $\Gamma$. Then
\begin{equation*}
H:=\{\chi\in X : m(\rho\bar\chi)=m(\rho)\}=\{\chi\in X : \varphi\chi=\varphi\}
\end{equation*}
is a subgroup of $X$; in particular $H$ is cyclic. Moreover, setting $N:=\bigcap_{\chi\in H}\ker\chi\subseteq\Gal(K/k)$ and $M:=K^{N}$, the extension $M/k$ is constant, $H$ is the character group of $\Gal(M/k)$, and $\prod_{\chi\in H}\operatorname{L}(\chi,s)=\zeta_M(s)$.
\end{corollary}

\begin{proof}
For $\chi\in X$ the conjugate $\bar\chi$ lies in $X$, and by Lemma~\ref{lem_mean_frob_real_nonneg}\textit{(2)} together with \eqref{eq:pointwise_stabiliser}, applied to $\bar\chi$, the condition $m(\rho\bar\chi)=m(\rho)$ is equivalent to $\varphi\bar\chi=\varphi$ pointwise; multiplying pointwise by $\chi$ and using $|\chi|=1$ shows that this is equivalent to $\varphi\chi=\varphi$. Thus $H$ is the stabiliser of $\varphi$ under the multiplication action of $X$, hence a subgroup, and it is cyclic as a subgroup of the character group of the cyclic group $\Gal(K/k)$. Viewing each $\chi\in H$ through $\Gal(K/k)$, duality for finite abelian groups identifies $H$ with the character group of $\Gal(K/k)/N=\Gal(M/k)$, and $M/k$ is constant as a sub-extension of $K/k$. Finally, equation  \eqref{eq:decomp_of_DEDEKIND_CONST_EXT}, applied to $M/k$, together with $\operatorname{L}(1,s)=\zeta_k(s)$ for the trivial character, which lies in $H$, gives $\prod_{\chi\in H}\operatorname{L}(\chi,s)=\zeta_M(s)$.
\end{proof}

\begin{definition}\label{def:algebraic-singularity}
A singularity of a complex function $f(z)$ at $z=w$ is called \emph{algebraic} if, in a neighbourhood of $w$ and for a fixed choice of branch, $f(z)$ can be written in the form
\begin{equation}
(1-z/w)^{\alpha} g(z),
\end{equation}
where $g(z)$ is analytic near $w$, $g(w)\neq 0$, and $\alpha \notin \{0,1,2,\ldots\}$. We call $-\alpha$ the \emph{order} of the singularity.
\end{definition}
%For $\lambda\in\CC$ and $\Real(s)>1$ we %set $\operatorname{L}(\chi,s)^{\lambda}:=\exp(\lambda\log\operatorname{L}(\chi,s))$, where $\log\operatorname{L}(\chi,s)$ is given by the absolutely convergent Dirichlet series $\sum_{v\in\Omega_k}\sum_{n\geq 1}\operatorname{tr}(\chi(\Frob_v)^n\mid V^{I_v})/(nq_v^{ns})$.%
\begin{proposition}\label{Prop:Frob_func_into_zeta}
Let $S$ be a finite set of places of $k$ and let $\rho$ be a real-valued and non-negative $S$-frobenian function. Assume that $|\rho(v)| < q_v$ holds for all $v \notin S$. Then the Euler product
\begin{equation}
F(s) = \prod_{v\notin S}\left(1+\frac{\rho(v)}{q_v^s}\right)
\end{equation}
has the form
\begin{equation}\label{eq:Ftozeta}
F(s) = \zeta_M(s)^{m(\rho)}\,P(s),\qquad \Real(s)>1,
\end{equation}
where $M/k$ is some constant extension and the function $P(s)$ admits an analytic continuation to $\Real(s) > \frac{1}{2}$, except for possible algebraic singularities located along $\Real(s) = 1$ that are distinct from those of $\zeta_M(s)$. Moreover, the real part of the order of singularities from $P(s)$ on $\Real(s) = 1$ is strictly less than $m(\rho)$. In particular, $P(s)$ is holomorphic and non-vanishing at $s=1$.
\end{proposition}

\begin{proof}
This proof is inspired by \cite[Prop.~2.3]{Frei_2022}. We first observe the Euler factors $1+\rho(v)q_v^{-s}$ are holomorphic on $\CC$ and non-zero for $\Real(s)\geq 1$, as $|\rho(v)|<q_v$ by assumption.

Next, recall that the irreducible characters of a finite group $\Gamma$ form a basis for the space of complex class functions of $\Gamma$ (\cite[Prop.~2.30]{FultonHarris1991}). In particular, if $\varphi:\Gamma\to\CC$ is the class function associated to $\rho$, then we may write
\begin{equation*}
\varphi = \sum_{\chi}\lambda_\chi\chi,
\end{equation*}
where the sum runs over the irreducible characters of $\Gamma$ and $\lambda_\chi = \langle\varphi,\chi\rangle = \frac{1}{|\Gamma|}\sum_{\gamma\in\Gamma}\varphi(\gamma)\overline{\chi(\gamma)} = m(\rho\bar\chi)$, by the orthonormality of the characters.

For $\Real(s) > 1$, we find that
\begin{equation*}
F(s) = \prod_{v\notin S}\left(1+\frac{\sum_\chi\lambda_\chi\chi(\Frob_v)}{q_v^s}\right) = P_1(s)\prod_{\chi}\operatorname{L}(\chi,s)^{\lambda_\chi},
\end{equation*}
where $P_1(s)$ is a holomorphic function with absolutely convergent Euler product on $\Real(s)>1/2$. Now group them by whether $\chi$ is constant or not:
\begin{equation}\label{Eq_Frob_decomp_into_constant_nonconst}
F(s) = P_2(s)\left(\prod_{\substack{\chi\\ \text{constant}}}\operatorname{L}(\chi,s)^{\lambda_\chi}\right)\left(\prod_{\substack{\chi\\ \text{non-constant}}}\operatorname{L}(\chi,s)^{\lambda_\chi}\right).
\end{equation}
Now, by Theorem~\ref{lemma_Sawin}, the third factor in \eqref{Eq_Frob_decomp_into_constant_nonconst} is holomorphic and non-vanishing for $\Real(s)>1/2$.
By \eqref{eq_Z(chiu)_is_shift}, the second factor in \eqref{Eq_Frob_decomp_into_constant_nonconst} becomes
\begin{equation}\label{Eq_second_factor_into_Z}
\prod_{\substack{\chi\\ \text{constant}}}Z(\chi(\Frob_q)u)^{\lambda_\chi}.
\end{equation}
Here, $Z(\chi(\Frob_q)u)$ is zero-free for $\Real(s)>1/2$ by the Riemann hypothesis (Lemma~\ref{lemma:value_zeta_function}) and $Z(\chi(\Frob_q)u)$ has poles of order $1$ along $\Real(s)=1$. Moreover, their location along $\Real(s)=1$ are distinct as the image of $\Frob_q$ under different $\chi$ is different, $\Frob_q$ being a generator of the cyclic group $\Gal(K/k)$.

By Lemma~\ref{lem_mean_frob_real_nonneg} \textit{(1)}, $|m(\rho\bar\chi)| \leq m(\rho)$. We take those $\chi$ with $\Real(m(\rho\bar\chi)) = m(\rho)$, group them together yields:
\begin{equation}\label{equ_Frob_proof_P3}
F(s) = P_3(s)\left(\prod_{\substack{\chi\ \text{constant}\\ m(\rho\bar\chi)=m(\rho)}}Z(\chi(\Frob_q)u)\right)^{m(\rho)},
\end{equation}
we have used Lemma~\ref{lem_mean_frob_real_nonneg} \textit{(2)}, $\Real(m(\rho\bar\chi))=m(\rho)$ if and only if $m(\rho\bar\chi)=m(\rho)$. Here $P_3(s)$ admits analytic continuation for $\Real(s)>\frac{1}{2}$, except for possible algebraic singularities along $\Real(s)=1$, which are disjoint from the second factor in \eqref{equ_Frob_proof_P3}. For the remaining constant $\chi$, the singularities of $Z(\chi(\Frob_q)u)^{\lambda_\chi}$ along $\Real(s)=1$ lie at the poles of $Z(\chi(\Frob_q)u)$, which are simple, and are algebraic of order $\lambda_\chi$ in the sense of Definition~\ref{def:algebraic-singularity}; moreover $\Real(\lambda_\chi)<m(\rho)$, since equality would give $m(\rho\bar\chi)=m(\rho)$ by Lemma~\ref{lem_mean_frob_real_nonneg} \textit{(2)} and place $\chi$ in the second factor. This proves the assertion on the orders.

By Corollary~\ref{cor:extremal_subgroup}, the characters appearing in the second factor of \eqref{equ_Frob_proof_P3} form a cyclic subgroup $\langle\chi_0\rangle$ of the group of constant characters, and the second factor equals the zeta function $\zeta_M(s)$ of an intermediate constant extension $k\subseteq M\subseteq K$.

It remains to prove the final assertion. A pole of $Z(\chi(\Frob_q)u)$ at $s=1$ would require $\chi(\Frob_q)=1$, hence $\chi$ trivial by the injectivity of $\chi\mapsto\chi(\Frob_q)$ noted above, and the trivial character appears in the second factor of \eqref{equ_Frob_proof_P3}; so $P(s)=P_3(s)$ is holomorphic at $s=1$. Moreover $P_1(1)\neq 0$, each non-constant factor $\operatorname{L}(\chi,1)^{\lambda_\chi}=\exp(\lambda_\chi\log\operatorname{L}(\chi,1))$ is non-zero, and for the remaining constant $\chi$, with $\alpha=\chi(\Frob_q)\neq 1$, the value $Z(\alpha q^{-1})=\mathcal L(\alpha q^{-1})/\bigl((1-\alpha q^{-1})(1-\alpha)\bigr)$, obtained from Lemma~\ref{lemma:value_zeta_function} at $q^{-s}=\alpha q^{-1}$, is finite and non-zero: $\mathcal L(\alpha q^{-1})\neq 0$ since $|\alpha q^{-1}|=q^{-1}$ while the roots of $\mathcal L$ have absolute value $q^{-1/2}$, and $(1-\alpha q^{-1})(1-\alpha)\neq 0$ since $|\alpha q^{-1}|<1$ and $\alpha\neq 1$. Hence $P(1)\neq 0$.
\end{proof}

\section{Dirichlet series and local Fourier transforms}
In this section, we develop the harmonic analysis over global function fields that will be used to study our Dirichlet series. Our main objective is to analyse the local Fourier transforms arising in this context, in particular their convergence properties and the singularities they contribute to the global Dirichlet series.
\subsection{Class field theory}For a subgroup $H\leq G$, the inclusion $H\subseteq G$ identifies $\Hom(\Gamma_k,H)$ with the set of $\varphi\in\Hom(\Gamma_k,G)$ whose image is contained in $H$; we write $F_{H,W}(s)$ for the series \eqref{DserieswithLC} formed with $H$-extensions of $k$, the function $f_W$ being evaluated through this inclusion, and we define
\begin{equation}
\widetilde F_{H,W}(s):=\sum_{\varphi \in \Hom(\Gamma_k,H)} \frac{f_{W}(\varphi)}{\Phi(\varphi)^s}.
\end{equation}
Here $\Phi(\varphi)$ does not depend on the target group, as it is determined by $\ker\varphi$. Every $\varphi\in\Hom(\Gamma_k,H)$ is surjective onto its image, so that $\widetilde F_{H,W}(s)=\sum_{J\leq H}F_{J,W}(s)$, and M\"obius inversion over the lattice of subgroups of $G$ gives
\begin{equation}\label{removesubgroups}
F_{G,W}(s)=\sum_{H\leq G}\mu(G/H)\,\widetilde F_{H,W}(s),
\end{equation}
where $\mu(G/H)$ denotes the M\"obius function of the interval $[H,G]$ in the lattice of subgroups of $G$; for a finite abelian group $A$, $\mu(A)=0$ unless $A$ has squarefree exponent, and $\mu(A)=\prod_p(-1)^{r_p}p^{\binom{r_p}{2}}$ when $A\cong\prod_pC_p^{r_p}$.  We will see in Proposition~\ref{which_cha_contri}\textit{(2)} that the terms with $H\lneq G$ in \eqref{removesubgroups} do not contribute to the leading terms of the asymptotic formulas. It therefore suffices to study $\widetilde F_{H,W}(s)$ for each subgroup $H\leq G$; as $H$ is a finite abelian group, we present the analysis for $\widetilde F_{G,W}(s)$, and the same applies to every subgroup.

Using the global Artin map $\mathbf{A}^* / k^* \rightarrow$ $\operatorname{Gal}\left(k^{\mathrm{ab}} / k\right)$,  we have the identification:
$$
\operatorname{Hom}(\operatorname{Gal}(\bar{k} / k), G)=\operatorname{Hom}\left(\mathbf{A}^* / k^*, G\right).
$$
It follows that
\begin{equation*}
    \sum_{\varphi \in \operatorname{sub}-G-\operatorname{ext}(k)} \frac{f_{W}(\varphi)}{\Phi(\varphi)^s} =  \sum_{\chi \in \operatorname{Hom}\left(\mathbf{A}^* / k^*, G\right)} \frac{f_{W}(\varphi)}{\Phi(\chi)^s},
\end{equation*}
where $\Phi(\chi)$ is the reciprocal of the idelic norm of the conductor of the kernel of $\chi$, which is the norm of the conductor of the sub-$G$-extension corresponding to $\chi$.

Note that we may also identify
$$
\operatorname{Hom}\left(\operatorname{Gal}\left(\bar{k}_v / k_v\right), G\right)=\operatorname{Hom}\left(k_v^*, G\right),
$$
using the local Artin map $ k_v^* \rightarrow \operatorname{Gal}\left(k_v^{\mathrm{ab}} / k_v\right)$.
\subsection{Harmonic analysis}
We start with a topological result on ideles of function fields:
\begin{lemma}\label{TopolocyofA}
Let $k$ be a global function field and let  $n \in \NN$. Then  $\AS^{*n}_k\subseteq\AS^*_k$ is closed. 
\end{lemma}
\begin{proof}
Recall the isomorphism between topological groups $\OO \cong (1+\mathfrak{m}_v)\times\F_v^*$, where $\mathfrak{m}_v$ is the maximal ideal of $\mathcal{O}_v$, \cite[Prop.~5.3.]{AlgNTNeu}. Therefore, to show $\mathcal{O}_v^{*n} \cong (1+\mathfrak{m}_v)^n\times\F_v^{*n}$ is closed in $\OO \cong (1+\mathfrak{m}_v)\times\F_v^*$ it suffices to show that $(1+\mathfrak{m}_v)^n$ is closed in $ 1+\mathfrak{m}_v$.

Consider the following continuous map
    $$
    (\cdot)^n:(1+\mathfrak{m}_v) \to (1+\mathfrak{m}_v),
    \\
    x \mapsto x^n,
    $$
    whose image $(1+\mathfrak{m}_v)^n $ is compact as $(1+\mathfrak{m}_v)$ is compact. 
Since $(1+\mathfrak{m}_v)$ is Hausdorff, it follows that $(1+\mathfrak{m}_v)^n$ is closed, hence $k_v^{*n} \subseteq k_v^*$ is closed

Now, consider the inclusion between topological spaces
    $$
    i: \left(\AS^*_k,\mathscr{T}_1\right) \to  \left(\prod_{v\in \Omega_k}k_v^*,\mathscr{T}_2\right),
    $$
    where $\mathscr{T}_1$ is the restricted product topology and $\mathscr{T}_2$ is the product topology. In particular, $\prod_{v\in \Omega_k}k_v^{*n}$ is closed in $\prod_{v\in \Omega_k}k_v^*$.
    Hence, $i^{-1}(\prod_{v\in \Omega_k}k_v^{*n}) = \AS_k^* \cap \prod_{v\in \Omega_k}k_v^{*n} = \AS_k^{*n}$ is closed in $\AS_k^*$.
\end{proof}  
\begin{remark}
    Lemma~\ref{TopolocyofA} may suggest that the topology of $k_v$ does not depends on the characteristic; however,  openness is  more subtle. 
    According to \cite{MOLubin}, if $K$ is a local field with characteristic $p$, then $K^{*n} \subseteq K^*$ is open if and only if $p \nmid n$.
\end{remark}
For a locally compact group $A$, we denote the Pontryagin dual of $A$ by $ A^{\wedge}:=\operatorname{Hom}\left(A, S^1\right)$. 
By Lemma \ref{TopolocyofA} and \cite[\S3.1.]{MR3884640}, we may identify the Pontryagin dual of $\operatorname{Hom}\left(\mathbf{A}^* / k^*, G\right)$ with $\mathbf{A}^* / k^* \otimes G^{\wedge}$.

We denote the associated pairing by $\langle\cdot, \cdot\rangle: \operatorname{Hom}\left(\mathbf{A}^* / k^*, G\right) \times\left(\mathbf{A}^* / k^* \otimes G^{\wedge}\right) \rightarrow S^1$. Similarly, the Pontryagin dual of $\operatorname{Hom}\left(k_v^*, G\right)$ is naturally identified with $k_v^* \otimes G^{\wedge}$, and we also denote the relevant Pontryagin pairing by $\langle\cdot, \cdot\rangle$. For each place $v$, we equip the finite group $\operatorname{Hom}\left(k_v^*, G\right)$ with the unique Haar measure $\mathrm{d} \chi_v$ such that
$$
\operatorname{vol}\left(\operatorname{Hom}\left(k_v^* / \mathcal{O}_v^*, G\right)\right)=1.
$$
Since $|\Hom(k_v^*/\OO, G)| = |G|$, this is $|G|^{-1}$ times the counting measure for any place. The product of these measures yields a well-defined measure $\mathrm{d} \chi$ on $\operatorname{Hom}\left(\mathbf{A}^*, G\right)$. 

For $x_v \in k_v^* \otimes G^{\wedge}$ we have the local Fourier transform
$$
\widehat{f}_{v, G}\left(x_v ; s\right)=\int_{\chi_v \in \operatorname{Hom}\left(k_v^*, G\right)} \frac{f_{v}(\chi_v)\left\langle\chi_v, x_v\right\rangle}{\Phi_v\left(\chi_v\right)^s} \mathrm{~d} \chi_v,
$$
where $\Phi_v\left(\chi_v\right)$ is the reciprocal of the $v$-adic norm of the conductor of Ker $\chi_v$. We use $\widehat{f}_v$ for shorthand when there is no ambiguity. If $G$ is tame, then $\widehat{f}_{v, G}\left(x_v ; s\right)$ is a finite sum. If $G$ is not tame, then it is not even clear whether $\widehat{f}_{v, G}\left(x_v ; s\right)$ converges absolutely. This is our first aim in the next subsection.
\subsection{Calculation of local Fourier transforms}
In this section, we will be using higher unit groups $\OOj = 1 + \pi^j_v\mathcal{O}$. We begin with the following descending filtration of higher unit groups
\begin{equation}\label{HigherUG}
    \cdots \subseteq \OOj = 1 + \pi^j_v\mathcal{O} \subseteq \cdots \subseteq  \OOtwo \subseteq  \OOone \subseteq  \OO.
\end{equation}
According to \cite[Chap.~III, §2]{Serre_1979}, these subgroups form a fundamental system of neighborhoods of $1$ in $\mathcal{O}_v^*$ with respect to the $\mathfrak {m}_v$-adic topology.  
This implies that $\bigcap_{j\geq 1} \OOj = \{1\}$, and $\bigcap_{j\ge 1}\mathfrak{m}^j_v=\{0\}$.
\begin{lemma}\label{trivialintersection}
Let $n \geq 1$. For each $j\ge 1$, the inclusion $\mathcal{O}_{v,j}^*\hookrightarrow \mathcal{O}_{v,1}^*$ induces a homomorphism
\[
\iota_{j,n}:\ \mathcal{O}_{v,j}^*/(\mathcal{O}_{v,j}^*)^{n}\ \longrightarrow\ \mathcal{O}_{v,1}^*/(\mathcal{O}_{v,1}^*)^{n}.
\]
Then
\[
\bigcap_{j\ge 1}\ \iota_{j,n}\!\big(\mathcal{O}_{v,j}^*/(\mathcal{O}_{v,j}^*)^{n}\big)\;=\;\{1\}\ \subset\ \mathcal{O}_{v,1}^*/(\mathcal{O}_{v,1}^*)^{n}.
\]
\end{lemma}

\begin{proof}
First assume that $n=p^t$ with $t\ge 1$.
For any $x\in \mathfrak m_v^{\,j}$ one has $(1+x)^{p^t}=1+x^{p^t}$, hence
$(\mathcal{O}_{v,j}^*)^{p^t}=1+\mathfrak m_v^{\,j p^t}.$
Let $[u]\in \mathcal{O}_{v,1}^*/(\mathcal{O}_{v,1}^*)^{p^t}$ lie in the image of $\iota_{j,p^t}$ for every $j\ge 1$.
Thus, for each $j$ there exists $v_j\in \mathcal{O}_{v,j}^*$ with $[u]=[v_j]$ in $\mathcal{O}_{v,1}^*/(\mathcal{O}_{v,1}^*)^{p^t}$, i.e.
$
u\,v_j^{-1}\in (\mathcal{O}_{v,1}^*)^{p^t}=1+\mathfrak m_v^{\,p^t}.
$
Write $v_j=1+a_j$ with $a_j\in \mathfrak m_v^{\,j}$ and $u\,v_j^{-1}=1+b_j$ with $b_j\in \mathfrak m_v^{\,p^t}$. Then
$
u=(1+a_j)(1+b_j)\in 1+\mathfrak m_v^{\,\min(j,p^t)}.
$
Since this holds for all $j$, taking any $j\ge p^t$ yields
$
u\in 1+\mathfrak m_v^{\,p^t}=(\mathcal{O}_{v,1}^*)^{p^t},
$
so $[u]=1$ in $\mathcal{O}_{v,1}^*/(\mathcal{O}_{v,1}^*)^{p^t}$.

If $\gcd(n,p)=1$,
we fix $j\ge 1$ and $a\in \mathcal{O}_{v,j}^*$. Consider $f(X)=X^n-a$ on $\mathcal{O}_{v,j}^*$.
We have $f(1)\equiv 0\pmod{\mathfrak m_v^{\,j}}$ and $f'(1)=n\in \mathcal{O}_v^{\times}$ (since $p\nmid n$).
By Hensel's lemma, the map $u\mapsto u^n$ is a bijection $\mathcal{O}_{v,j}^*\!\to\!\mathcal{O}_{v,j}^*$.
Hence $(\mathcal{O}_{v,j}^*)^{n}=\mathcal{O}_{v,j}^*$ and the quotient
$\mathcal{O}_{v,j}^*/(\mathcal{O}_{v,j}^*)^{m}$ is trivial for every $j$; consequently the intersection of their images is $\{1\}$.

The assertion now follows as $\mathcal{O}_{v,j}^*/(\mathcal{O}_{v,j}^*)^{mn} \cong \mathcal{O}_{v,j}^*/(\mathcal{O}_{v,j}^*)^{m} \times \mathcal{O}_{v,j}^*/(\mathcal{O}_{v,j}^*)^{n}$ whenever $\gcd(m,n)=1.$
\end{proof}
Furthermore, we have the following result regarding the structure of $\OOone/\OOj$ from page 17 of \cite{MR1785410}.
\begin{proposition}\label{quotientofhigherunitgroup}
 Let $v$ be a place of $k, \pi_v \in k_{v}$ a uniformizer at $v$, such that $B_v$ as $\mathbb{F}_p$-basis of $\mathbb{F}_{v}$ and $n \in \mathbb{N}$.
For a rational number $u  \in \QQ$, let 
$$
\left\lceil u \right\rceil_p := \min \left\{l : u \leq p^l, l \in \mathbb{N}_0\right\}= \begin{cases}
    \left\lceil \log_p  u \right\rceil, &u \geq 1,\\
    0 ,  &u \leq 1.
\end{cases}
$$
Moreover, denote by $\mathbb{N}_p^*:=\mathbb{N} \backslash p \mathbb{Z}$ the set of positive integers prime to $p$.
Then there is a non-canonical isomorphism of finite p-groups
$$
\begin{aligned}
\mu^{(j)}:  \OOone / \OOj &\rightarrow \prod_{t \in \mathbb{N}_p^*}\left(\mathbb{Z} / \mathbb{Z}p^{\left\lceil\frac{j}{t}\right\rceil_p}\right)^{B_v}, \\
\prod_{\substack{t \in \mathbb{N}_p^* \\
\beta \in B_v}}\left(1+\beta \pi^t\right)^{m_{t \beta}} \OOj &\mapsto\left(m_{t \beta}+{\mathbb{Z}p^{\left\lceil\frac{j}{t}\right\rceil_p}}\right)_{\substack{t \in \mathbb{N}_p^* \\
\beta \in B_v}}.
\end{aligned}
$$
\end{proposition}

\begin{lemma}\label{absconvpre}
    Let $v \in \Omega_k$ and $x_v \in k_v^* \otimes 
    G^{\wedge}$. Then 
    \begin{equation*}
        |\widehat{f}_v\left(x_v ; s\right)|
        \leq 1 + \frac{|\Hom (\F_{q_v}^*, G_t)| - 1 }{q_v^{\operatorname{Re}(s)}} + \sum_{j = 2}^{\infty}  \frac{|\Hom(\OO/\OOj,G)|-|\Hom(\OO/\OOjs,G)|}{q_v^{j\operatorname{Re}(s)}}.
    \end{equation*}
\end{lemma}
\begin{proof}
    By our choice of measures, we have
\begin{equation}\label{localFourier}
    \widehat{f}_{W_v}\left(x_v ; s\right)=\frac{1}{|G|} \sum_{\chi_v \in \operatorname{Hom}\left(k_v^*, G\right)} \frac{f_{W_v}(\varphi_v)\left\langle\chi_v, x_v\right\rangle}{\Phi_v\left(\chi_v\right)^s}.
\end{equation}
Taking absolute value of (\ref{localFourier}) yields
\begin{align*}
 |\widehat{f}_v\left(x_v ; s\right)|
        \leq \frac{1}{|G|} \sum_{\chi_v \in \operatorname{Hom}\left(k_v^*, G\right)} \frac{1}{\left|\Phi\left(\chi_v\right)^{\Real(s)}\right|}.
    \end{align*}
    Recall, after choosing a uniformizer, we get
\begin{equation}\label{kvOvses}
\operatorname{Hom}\left(k_v^*, G\right) \cong \operatorname{Hom}\left(k_v^* / \mathcal{O}_v^*, G\right) \oplus \operatorname{Hom}\left(\mathcal{O}_v^*, G\right) . 
\end{equation}
Hence
\begin{align*}
|\widehat{f}_v\left(x_v ; s\right)| \leq \frac{1}{|G|} \sum_{\varphi_v \in \operatorname{Hom}\left(k_v^*/\mathcal{O}_v^*, G\right)} \sum_{\psi_v \in \operatorname{Hom}\left(\mathcal{O}_v^*, G\right)} \frac{1}{\left|\Phi\left(\varphi_v\psi_v\right)^{Re(s)}\right|}.
\end{align*}
Furthermore, filtering these characters by their conductor yields
   \begin{align}\label{lftb}
        |\widehat{f}_v\left(x_v ; s\right)|
        \leq \sum_{j = 0}^{\infty} \sum_{\substack{ \psi_v \in \operatorname{Hom}\left(\mathcal{O}_v^*, G\right) \\ \Phi(\psi_v) = q^j}} \frac{1}{q_v^{jRe(s)}},
   \end{align}
as $|\Hom(\kk/\OO,G)|= |G|$. We observe that
\begin{align*}
|\left\{ \psi_v \in \Hom(\mathcal{O}_v^*, G): \Phi(\psi_v) = q_v^{j} \right\}|&=|\left\{ \psi_v \in \Hom(\mathcal{O}_v^*, G): \Phi(\psi_v)\leq q_v^j \right\}|
\\ & -|\left\{ \psi_v \in \Hom(\mathcal{O}_v^*, G): \Phi(\psi_v)\leq q_v^{j-1} \right\}|.
\end{align*} 
The result follows after identifying 
$\left\{ \psi_v \in \Hom(\mathcal{O}_v^*, G): \Phi(\psi_v)\leq q_v^j \right\}$ with $\Hom(\OO/\OOj,G)$.
\end{proof}
   
Therefore, to prove the absolute convergence of $\widehat{f}_v\left(x_v ; s\right)$, we count $|\Hom(\OO/\OOj,G)|$.

\begin{proposition}\label{MainCount}
Recall $G_w = \prod_{m=1}^e C_{p^m}^{r_m}$, we have
\begin{equation}\label{MainCountEqu}
|\Hom(\mathcal{O}_v^*/\OOj, G)| = q_v^{\sum^e_{m=1}r_m(G)\left(j-\left\lceil\frac{j}{p^m} \right\rceil\right)}\times|\Hom(\F_{q_v}^*, G_t)| .
\end{equation}
\end{proposition}
\begin{proof}
    Recall the filtration from (\ref{HigherUG}). By Proposition \ref{quotientofhigherunitgroup}, we have $\OO/\OOone \cong \F_{q_v}^* $ and for $j \geq 1 $, 
 $\OOj/\OOjj \cong \pi_v^j \OOO/\pi_v^{j+1} \OOO\cong \F_{q_v}$.  The first isomorphism induces the following short exact sequence
\begin{equation}\label{lses}
        1 \to \OOone \to \OO \to \F_{q_v}^* \to 1,
\end{equation}
this sequence splits by applying Hensel's lemma to $ f(x) = x^{q_v -1} -1$ since $\F_{q_v}^*$ are the simple roots of $f$ mod $v$. After taking $\Hom(-,G)$, we have
    $$
    1 \to \Hom(\OOone,G) \to \Hom(\OO,G) \to \Hom(\F_q^* ,G) \to 1.
    $$
Together with the decomposition $\mathcal{O}_v^*/\OOj \cong \F_{q_v}^* \times \OOone/\OOj$, we deduce that
\begin{align*}\label{mixdec}
|\operatorname{Hom}\left(\mathcal{O}_v^*/\OOj, G\right)| &= |\Hom (\F_{q_v}^*, G_w)|\times |\Hom (\OOone/\OOj, G_w)| \\ &\times |\Hom (\F_{q_v}^*, G_t)| \times |\Hom (\OOone/\OOj, G_t)|. 
\end{align*}
However,
\begin{align*}
|\operatorname{Hom}\left(\mathcal{O}_v^*/\OOj, G\right)|  =  |\Hom(\OOone/\OOj, G_w)| \times |\Hom (\F_{q_v}^*, G_t)| ,
\end{align*}
as $|\F_{q_v}^*| = q^{\operatorname{deg}(v)}-1$ and  $|\Hom (\OOone/\OOj, G_t)| =1$.

Using Proposition \ref{quotientofhigherunitgroup}, we can compute $|\Hom(\OOone/\OOj, G_w)|$ explicitly. 
Indeed,
\begin{align*}
    & |\Hom(\OOone/\OOj, G_w)| = |\prod^e_{m=1}\Hom(\OOone/\OOj, C_{p^m}^{r_m(G)})| = \prod^e_{m=1}|\Hom(\OOone/\OOj, C_{p^m})|^{r_m(G)}.
\end{align*}
Furthermore, we have
\begin{align*}
    |\Hom(\OOone/\OOj, C_{p^m})| = \prod_{t\in \NN^*_p}|\Hom(\ZZ/\ZZ{p^{\left\lceil\frac{j}{t}\right\rceil_p}}, C_{p^m})|^{B} = \prod_{\substack{t =1 \\ t\in \NN^*_p \\\left\lceil \log_p \frac{j}{t}\right\rceil_p \geq m}}^{j}q_v^m\prod_{\gamma = 1}^{m-1} \prod_{\substack{t = 1 \\ t \in \NN_p^* \\\left\lceil \log_p \frac{j}{t}\right\rceil_p = \gamma }}^{j} q_v^{\gamma}, 
\end{align*}
as
\begin{equation}
    |\Hom(\ZZ/\ZZ{p^{\left\lceil\frac{j}{t}\right\rceil_p}}, C_{p^m})| = \begin{cases} 
        p^{\left\lceil\frac{j}{t}\right\rceil_p},& {\left\lceil\frac{j}{t}\right\rceil_p <m,} \\
p^{m},& {\left\lceil\frac{j}{t}\right\rceil_p  \geq m.}
         \end{cases}
\end{equation}
Now, since  $ A(m):=|\{t: \left\lceil \log_p \frac{j}{t} \right\rceil\geq m\}| = \left\lceil \frac{j}{p^{m-1}} -1 \right\rceil - \left\lceil \frac{j}{p^{m}} -1 \right\rceil, $ we have 
\begin{align*}
    \prod_{\substack{t =1 \\ t\in \NN^*_p \\\left\lceil \log_p \frac{j}{t}\right\rceil_p \geq m}}^{j}q_v^m\prod_{\gamma = 1}^{m-1} \prod_{\substack{t = 1 \\ t \in \NN_p^* \\\left\lceil \log_p \frac{j}{t}\right\rceil_p = \gamma }}^{j} q_v^{\gamma} &= q_v^{mA(m) + \sum_{\gamma=1}^{m-1}\gamma (A(\gamma)-A(\gamma +1))} \\&= q_v^{ \sum_{\gamma=1}^{m}A(\gamma)} = q_v^{j-1 -\left\lceil\frac{j}{p^{m}}-1\right\rceil}= q_v^{j -\left\lceil\frac{j}{p^{m}}\right\rceil},
\end{align*}
where  the last equality is due to $\sum_{\gamma=1}^{m}A(\gamma)$ is a telescoping series and $ \sum_{\gamma=1}^{m-1}\gamma (A(\gamma)-A(\gamma +1))$ is a variant of a telescoping series.
Therefore, the result follows.
\end{proof}

Now we are in the position to show the absolute convergence of local Fourier transforms.

\begin{proposition}\label{absconvergence:delta} Let $\delta > 1$, $ v\in \Omega_k$   and $x_v\in k_v^* \otimes G^{\wedge}$. Then $\widehat{f}_v\left(x_v ; s\right)$ converges absolutely for $\operatorname{Re}(s) \geq a(G)- \frac{1}{\delta p^e}$, where $e$ is the exponent of $G_w$.
\end{proposition}
\begin{proof}
By Lemma \ref{absconvpre}, we have
    \begin{align*}
    |\widehat{f}_v\left(x_v ; s\right)|
     \leq  1 &+ \frac{|\Hom (\F_{q_v}^*, G_t)| - 1 }{q_v^{\operatorname{Re}(s)}} 
    \\ &+ \sum_{i = 2 }^{\infty} \frac{ (|\OOone/\OOj, G_w)| 
 - |\OOone/\OOjs, G_w)|)\times |\Hom (\F_{q_v}^*, G_t)| }{q_v^{j\operatorname{Re}(s)}}.
\end{align*}
Using Proposition \ref{MainCount}, we conclude that
\begin{align*}
& |(\OOone/\OOj, G_w)| 
 - |(\OOone/\OOjs, G_w)|
\\&= q_v^{\sum^e_{m=1}r_m(G)\left(j-1 -\left\lceil\frac{j-1}{p^m} \right\rceil\right)}\left(q_v^{\sum^e_{m=1}r_m(G)\left(j-\left\lceil\frac{j}{p^m} \right\rceil \right)
 - \sum^e_{m=1}r_m(G)\left(j-1 -\left\lceil\frac{j-1}{p^m} \right\rceil\right)}-1\right)
 \\ & = q_v^{\sum^e_{m=1}r_m(G)\left(j-1 -\left\lceil\frac{j-1}{p^m} \right\rceil\right)}
\left(q_v^{\sum_{m=1}^er_m(G)\left(1-\left\lceil \frac{j}{p^m}-1\right\rceil +\left\lceil \frac{j-1}{p^m}-1\right\rceil\right)}-1\right)
\\ & = q_v^{\sum^e_{m=1}r_m(G)\left(j-1 -\left\lceil\frac{j-1}{p^m} \right\rceil\right)}
\left(q_v^{\sum^e_{m=1+v_p(j-1)}r_m(G)}-1\right).
\end{align*}
Hence,
\begin{align*}
|\widehat{f}_{ v}\left(x_v ; s\right)|\leq 1 &+ \frac{|\Hom (\F_{q_v}^*, G_t)| - 1 }{q_v^{s}}  \\& + |\Hom (\F_{q_v}^*, G_t)|   \times \sum^{\infty}_{\substack{j=2\\v_p(j-1)\leq e-1}} \frac{\left(q_v^{\sum^e_{m=1+v_p(j-1)}r_m(G)}-1\right)\left(q_v^{\sum_{m=1}^er_m(G)\left(j-1-\left\lceil\frac{j-1}{p^m}\right\rceil\right)}\right)}{q_v^{js}}.
\end{align*}
Note that we may simplify the above expression by observing that 
\begin{align*}
 \sum_{m=1}^er_m(G)\left(j-1-\left\lceil\frac{j-1}{p^m}\right\rceil\right)+\sum_{l=v_p(j-1)+1}^{e}r_m(G) = \sum_{m=1}^er_m(G)\left(j-1-\left\lfloor\frac{j-1}{p^m}\right\rfloor\right),
\end{align*}
as $\left\lceil x \right\rceil - \left\lfloor  x \right\rfloor = 0$ when $x \in \Z$, and $\left\lceil x\right\rceil - \left\lfloor x \right\rfloor = 1$ when $x \notin \Z$. Therefore,
\begin{align}
     |\widehat{f}_{ v}\left(x_v ; s\right)| = 1&+ \frac{|\Hom (\F_{q_v}^*, G_t)| - 1 }{q_v^{\operatorname{Re}(s)}} \label{computelocalFourier} \\ &+  |\Hom (\F_{q_v}^*, G_t)|  \times\sum^{\infty}_{\substack{j=2\\v_p(j-1)\leq e-1}} \frac{q_v^{\sum_{m=1}^er_m(G)\left(j-1-\left\lfloor\frac{j-1}{p^m}\right\rfloor\right)}-q_v^{\sum_{m=1}^er_m(G)\left(j-1-\left\lceil\frac{j-1}{p^m}\right\rceil\right)}}{q_v^{j\operatorname{Re}(s)}}. \nonumber
\end{align}
If $\Real (s) \geq a(G) - \frac{1}{\delta p^e}$, then the largest $q_v$-exponent from each term of (\ref{computelocalFourier}) is 
\begin{align*}
-j\Real (s) + {\sum_{m=1}^er_m(G)\left(j-1-\left\lfloor\frac{j-1}{p^m}\right\rfloor\right)} &\leq -ja(G) + \frac{j}{\delta p^e}+{\sum_{m=1}^er_m(G)\left(j-1-\left\lfloor\frac{j-1}{p^m}\right\rfloor\right)} \\ &=
-\frac{j}{p^e} + \frac{j}{\delta p^e}+{\sum_{m=1}^er_m(G)\left(\frac{j}{p^m}-1-\left\lfloor\frac{j-1}{p^m}\right\rfloor\right)}.
\end{align*}
If $p^m \nmid j-1,$ then $\frac{j}{p^m}-1-\left\lfloor\frac{j-1}{p^m}\right\rfloor = \frac{j}{p^m}-\left\lceil\frac{j-1}{p^m}\right\rceil = \frac{j}{p^m}-\left\lceil\frac{j}{p^m}\right\rceil \leq 0.$ Otherwise, if $p^m \mid j-1,$ we have that $\frac{j}{p^m}-1-\left\lfloor\frac{j-1}{p^m}\right\rfloor = \frac{j}{p^m}-1-\frac{j-1}{p^m}<0$. 
Since $-\frac{j}{p^e} + \frac{j}{\delta p^e} <0$ the convergence of (\ref{computelocalFourier}) follows.
\end{proof}

We next compute local Fourier transforms for arbitrary $x_v\in k_v^* \otimes G^{\wedge}$. 
\begin{lemma}\label{computelocalinOv}
Recall that $S  \subseteq \Omega_k$ is a finite set of places. Let $v \notin S$ and let $x_v \in k_v^* \otimes G^{\wedge}$. Then
$$
%no need to add S as constant function?
\widehat{f}_v\left(x_v ; s\right)= \begin{cases}\sum_{\chi_v \in \operatorname{Hom}\left(\mathcal{O}_v^*, G\right)} \frac{f_{W_v}\left(\chi_v\right)\left\langle\chi_v, x_v\right\rangle}{\Phi_v\left(\chi_v\right)^s}, & \text { if } x_v \in \mathcal{O}_v^* \otimes G^{\wedge} \\ 0, & \text { otherwise.}\end{cases}
$$
\end{lemma}

\begin{proof}
From (\ref{kvOvses}),we have
$$
\widehat{f}_v\left(x_v ; s\right)=\frac{1}{|G|} \sum_{\psi_v \in \operatorname{Hom}\left(k_v^* / \mathcal{O}_v^*, G\right)} \sum_{\chi_v \in \operatorname{Hom}\left(\mathcal{O}_v^*, G\right)} \frac{f_{W_v}\left(\psi_v \chi_v\right)\left\langle\psi_v \chi_v, x_v\right\rangle}{\Phi_v\left(\chi_v\right)^s}.
$$
Note that outside of $S$, the function $f_{W_v}$ only takes the value $1$. Therefore
$$
\widehat{f}_{W_v, G}\left(x_v ; s\right)=\frac{1}{|G|} \sum_{\chi_v \in \operatorname{Hom}\left(\mathcal{O}_v^*, G\right)} \frac{f_{W_v}(\varphi_v)\left\langle\chi_v, x_v\right\rangle}{\Phi_v\left(\chi_v\right)^s} \sum_{\psi_v \in \operatorname{Hom}\left(k_v^* / \mathcal{O}_v^*, G\right)}\left\langle\psi_v, x_v\right\rangle .
$$
Now, character orthogonality gives
$$
\sum_{\psi_v \in \operatorname{Hom}\left(k_v^* / \mathcal{O}_v^*, G\right)}\left\langle\psi_v, x_v\right\rangle= \begin{cases}\left|\operatorname{Hom}\left(k_v^* / \mathcal{O}_v^*, G\right)\right|, & \text { if } x_v \in \mathcal{O}_v^* \otimes G^{\wedge}, \\ 0, & \text { otherwise. }\end{cases}
$$
Indeed, the subgroup $\mathcal{O}_v^* \otimes G^{\wedge} \subseteq k_v^* \otimes G^{\wedge}$ is naturally identified with the Pontryagin dual of $\operatorname{Hom}\left(k_v^* / \mathcal{O}_v^*, G\right)$. The result now follows as $\left|\operatorname{Hom}\left(k_v^* / \mathcal{O}_v^*, G\right)\right| = \left|\operatorname{Hom}\left(\Z, G\right)\right|=|G|$.
\end{proof}
Now we state the local Fourier transform for all cases. Recall by Lemma~\ref{trivialintersection}, we have that $\bigcap_{j\geq 1}\mathcal{O}^*_{v,j} \otimes G^{\wedge} = \{1\}.$ 
\begin{lemma}\label{LocalFT}
Let $v \notin S$ and let $x_v \in \OO \otimes G^{\wedge}$. 

If $G_w$ is non-trivial,
we have the following results for $\widehat{f}_{ v}\left(x_v ; s\right).$
\begin{enumerate}
\item For $x_v \notin\mathcal{O}_{v,1}^*\otimes G^{\wedge} $,
\begin{align*}
\widehat{f}_{ v}\left(x_v ; s\right)=1+ \frac{ - 1}{q_v^{s}}.
\end{align*}
    \item 
    For $x_v \in \mathcal{O}_{v,1}^*\otimes G^{\wedge} \text{ and }  x_v \notin \mathcal{O}_{v,2}^*\otimes G^{\wedge},$

\begin{align*}
\widehat{f}_{ v}\left(x_v ; s\right)=1+ \frac{|\Hom (\F_{q_v}^*, G_t)| - 1 }{q_v^{s}} - \frac{|\Hom (\F_{q_v}^*, G_t)|}{q_v^{2s}}.
\end{align*}
    \item Let $h\geq 2$: 
For $x_v \in \mathcal{O}_{v,h}^*\otimes G^{\wedge} \text{ and }  x_v \notin \mathcal{O}_{v,h+1}^*\otimes G^{\wedge}$,

\begin{align}
\widehat{f}_{ v}\left(x_v ; s\right)=1 & + \frac{|\Hom (\F_{q_v}^*, G_t)| - 1 }{q_v^{s}} \label{computelocalFourierfordifferentchar} \\&+ |\Hom (\F_{q_v}^*, G_t)|  \times\sum^{h}_{\substack{j=2\\v_p(j-1)\leq e-1}} \frac{q_v^{\sum_{m=1}^er_m(G)\left(j-1-\left\lfloor\frac{j-1}{p^m}\right\rfloor\right)}-q_v^{\sum_{m=1}^er_m(G)\left(j-1-\left\lceil\frac{j-1}{p^m}\right\rceil\right)}}{q_v^{js}}
\nonumber \\ &- |\Hom (\F_{q_v}^*, G_t)| \times\frac{q_v^{\sum^e_{m=1}r_m(G)\left(h-\left\lceil\frac{h}{p^m} \right\rceil\right)}}{q_v^{(h+1)s}}. \nonumber
\end{align}
\item 
If $x_v = \mathbf{1}_v$ is trivial, we have 
\begin{align}
    \widehat{f}_{ v}\left(x_v ; s\right) = 1&+ \frac{|\Hom (\F_{q_v}^*, G_t)| - 1 }{q_v^{s}} \label{trivialcharLocalFourierTransform} \\ &+  |\Hom (\F_{q_v}^*, G_t)|  \times\sum^{\infty}_{\substack{j=2\\v_p(j-1)\leq e-1}} \frac{q_v^{\sum_{m=1}^er_m(G)\left(j-1-\left\lfloor\frac{j-1}{p^m}\right\rfloor\right)}-q_v^{\sum_{m=1}^er_m(G)\left(j-1-\left\lceil\frac{j-1}{p^m}\right\rceil\right)}}{q_v^{js}}. \nonumber
\end{align}
\end{enumerate}
If $G_w$ is trivial, according to Lemma \ref{trivialintersection},  $\mathcal{O}_{v,j}^*\otimes G^{\wedge} $ is trivial for any $j$. In this case, we have 
\begin{equation}\label{purelytameFT}
\widehat{f}_v\left(x_v ; s\right)= \begin{cases}1+ \frac{|\Hom (\F_{q_v}^*, G_t)| - 1 }{q_v^{s}} , & \text { if } x_v \in \mathcal{O}_v^* \otimes G^{\wedge} \text{ is trivial},
\\ 1- \frac{1}{q_v^{s}} , & \text { otherwise.}\end{cases}
\end{equation}
\end{lemma}
\begin{proof}
 The proof for the purely tame case follows immediately from \cite[Lem.~3.10] {Frei_2022} by  taking $\mathcal{A} = \{1\}$ in their setting. Now we assume that $G_w$ is non-trivial.  
Lemma \ref{computelocalinOv} shows that for $v \notin S$,
\begin{equation}\label{localFouriersimplified}
\widehat{f}_{ v}\left(x_v ; s\right)= \begin{cases}\left\langle1, x_v\right\rangle +  \sum_{j=1}^\infty \frac{\sum_{\substack{\psi_v \in \operatorname{Hom}\left(\mathcal{O}_v^*, G\right)\\ \Phi(\psi)= q_v^j}} \left\langle\psi_v, x_v\right\rangle}{q_v^{js}} , &\text { if } x_v \in \mathcal{O}_v^* \otimes G^{\wedge}, \\ 0 ,& \text { if } x_v \notin \mathcal{O}_v^* \otimes G^{\wedge}.\end{cases}
\end{equation}
\newline
Since $\sum_{\substack{\psi_v \in \operatorname{Hom}\left(\mathcal{O}_v^*, G\right)\\ \Phi(\psi)= q_v^j}} \left\langle\psi_v, x_v\right\rangle = \sum_{\substack{\psi_v \in \operatorname{Hom}\left(\mathcal{O}_v^*/\OOj, G\right)}} \left\langle\psi_v, x_v\right\rangle - \sum_{\substack{\psi_v \in \operatorname{Hom}\left(\mathcal{O}_v^*/\OOjs, G\right)}} \left\langle\psi_v, x_v\right\rangle$,
by character orthogonality, we have
that, for $j \geq 2$, the sum $\sum_{\substack{\psi_v \in \operatorname{Hom}\left(\mathcal{O}_v^*, G\right)\\ \Phi(\psi)= q_v^j}} \left\langle\psi_v, x_v\right\rangle$ equals
\newline
\begin{align*}
    \begin{cases}
        |\left\{ \psi_v \in \Hom(\mathcal{O}_v^*, G): \Phi(\psi_v) = q_v^{j} \right\}|, &\text{if } x_v \in\OOj\otimes G^{\wedge}, \\
        - |\left\{ \psi_v \in \Hom(\mathcal{O}_v^*, G): \Phi(\psi_v)\leq q_v^{j-1} \right\}|, &\text{if }  x_v \in \OOjs\otimes G^{\wedge} \text{ and }  x_v \notin \OOj\otimes G^{\wedge} ,
        \\
        0, &\text{otherwise}.
    \end{cases}
\end{align*}
Moreover, Proposition \ref{MainCount} shows that if $j \geq 2$ we have:
\begin{align*}\label{caseforterms}
&\frac{\sum_{\substack{\psi_v \in \operatorname{Hom}\left(\mathcal{O}_v^*, G\right)\\ \Phi(\psi)= q_v^j}} \left\langle\psi_v, x_v\right\rangle }{|\Hom (\F_{q_v}^*, G_t)|} = 
    \\& \begin{cases}
q_v^{\sum^e_{m=1}r_m\left(j-1 -\left\lceil\frac{j-1}{p^m} \right\rceil\right)}
\left(q_v^{\sum^e_{m=1+v_p(j-1)}r_m}-1\right),&\text{} x_v \in \OOj\otimes G^{\wedge}, \\
    -q_v^{\sum^e_{m=1}r_m(G)\left(j-1-\left\lceil\frac{j-1}{p^m} \right\rceil\right)},&\text{}  x_v \in \OOjs\otimes G^{\wedge} \text{ and }  x_v \notin \OOj\otimes G^{\wedge} ,
        \\
        0, &\text{otherwise}.
    \end{cases}
\end{align*}
 After substituting the preceding result into (\ref{localFouriersimplified}) the assertion follows from the observation that when $v_p(j-1) \geq e$, we have $q_v^{\sum^e_{m=1+v_p(j-1)}r_m(G)} = 1$.
\end{proof}
To understand $\widehat{f}_{ v}$, we perform some calculations in the next lemma.
\begin{lemma}\label{LocationofPoleComputation}
    Let $j \geq 2$. Define 
    $$ z_j(G) = \frac{1 + \sum_{m=1}^er_m(G)\left(j-1-\left\lfloor\frac{j-1}{p^m}\right\rfloor\right) }{j}, \text{ } y_j(G) = \frac{1 + \sum_{m=1}^er_m(G)\left(j-1-\left\lceil\frac{j-1}{p^m}\right\rceil\right) }{j}.$$
    Then 
    \begin{enumerate}
    \item 
        $y_j(G) \leq z_j(G)$ for $j \geq 2 $.
        \item 
        $y_j(G) < z_j(G)$ for $p^f\mid j$\text{ with } $2 \leq j \leq p^e $.
        \item There exists $\varepsilon_j(G) >0$ such that
        $z_{j}(G)+ \varepsilon_j(G)=z_{p^e}(G) = a(G)$ for $  j \geq p^e +1$.
        \item 
        $1\leq z_{p^e}(G)$ with equality holds if and only if $G_w$ cyclic.
        \item 
        If $H_w \subsetneq G_w$, $z_{p^e}(H) < z_{p^e}(G).$
    \end{enumerate}
\end{lemma}
\begin{proof}
    \textit{(1)} follows from $\lceil x\rceil\geq \lfloor x \rfloor$.
    Now we prove \textit{(2)}. If $f = 0$, we have $y_j(G) = \frac{p^e-1}{p^e}<z_j(G) = 1.$ If $f >0,$ we have $r_f(G) >0$ and $p^f \nmid j -1$. 
    Hence $\left\lfloor\frac{j-1}{p^m}\right\rfloor < \left\lceil\frac{j-1}{p^m}\right\rceil$.
    To prove \textit{(3)},
    let $k \geq 2$ and $ 2 \leq w \leq p^e -1$, then
\begin{equation*}
    z_{p^e}(G)-z_{w+kp^e}(G)= 
    \frac{w+(k-1)p^e+p^e\sum_{m=1}^{e}r_m(G)\left(\left\lfloor\frac{w-1}{p^m}\right\rfloor - \frac{w}{p^m}+1\right)}{p^e(w+kp^e)} > 0,
\end{equation*}
as $\left(\left\lfloor\frac{w-1}{p^m}\right\rfloor - \frac{w}{p^m}+1\right) \geq 0$ and $w+(k-1)p^e > 0$. 
To prove \textit{(4)}, we observe that $z_{p^e}(G)=\frac{1+ \sum_{m=1}^er_m(G)(p^e- p^{e-m})}{p^e} $. If $G_w$ is cyclic, we have $z_{p^e}=1$. If $G_w$ is not cyclic, then $\sum_{m=1}^er_m(G)(p^e- p^{e-m}) > p^e -1 $ as $\sum_{m=1}^er_m(G) >1$.
Finally, \textit{(5)} follows from the definition of $z_j(G)$ and $y_j(G)$.
\end{proof}
\begin{lemma}\label{asympoticofTriviallocalft} 
There exists $\xi >0$ such that, for $\Real(s) \geq a(G)-\xi,$      
\begin{align}
      \widehat{f}_{ v}\left(x_v ; s\right) = 1&+ \frac{|\Hom (\F_{q_v}^*, G_t)| - 1 }{q_v^{s}} \label{TrivialLFTBigO} \\ &+  |\Hom (\F_{q_v}^*, G_t)|  \times\sum^{p^e}_{\substack{j=2\\v_p(j-1)\leq e-1}} \frac{q_v^{\sum_{m=1}^er_m(G)\left(j-1-\left\lfloor\frac{j-1}{p^m}\right\rfloor\right)}-q_v^{\sum_{m=1}^er_m(G)\left(j-1-\left\lceil\frac{j-1}{p^m}\right\rceil\right)}}{q_v^{js}} \nonumber \\ &+ O\left(\frac{1}{q_v^{\frac{s}{a(G)}+\xi}}\right). \nonumber
\end{align}
\end{lemma}
\begin{proof}
    Let $u>0$. We denote $\sum_{m=1}^er_m(G)\left(j-1-\left\lfloor\frac{j-1}{p^m}\right\rfloor\right)$ by $\sigma_j(G)$. By Proposition \ref{LocationofPoleComputation}, we have
    \begin{equation*}
        \frac{ju}{1+\sum_{m=1}^er_m(G)\left(j-1-\left\lfloor\frac{j-1}{p^m}\right\rfloor\right)} = \frac{u}{a(G)-\varepsilon_j(G)}.
    \end{equation*}
Therefore,
    \begin{equation}\label{rhs}
        ju-\sigma_j(G) -\frac{u}{a(G)} = \frac{\varepsilon_j(G)ju}{a(G)} - \left(1- \frac{u}{a(G)} \right)\sigma_j(G).
    \end{equation}
we observe that, for $\delta >1$,
$$\min_{u \geq a(G) -\frac{1}{\delta p^e}} \left\{\frac{\varepsilon_j(G)ju}{a(G)} - \left(1- \frac{u}{a(G)}  \right)\sigma_j(G)\right\} =\frac{1}{a(G)}\left(\varepsilon_j(G)ja(G)-\frac{\sigma_j(G)+1}{\delta p^e}\right ).$$ Hence, the right-hand side of (\ref{rhs}) is positive by choosing suitably large $\delta$. Recall from Proposition \ref{absconvergence:delta}, $\widehat{f}_v\left(x_v ; s\right)$ converges absolutely for $\operatorname{Re}(s) \geq a(G)- \frac{1}{\delta p^e}$ with any $\delta >1$. Moreover, our calculation here shows that there exist $\delta$ such that the exponent of $q_v$ is no more than $ \frac{s}{a(G)}$ with $\Real(s)\geq a(G)- \frac{1}{\delta p^e}$.
\end{proof}

\section{Global Fourier transforms and Poisson summation}
\subsection{Global Fourier transforms}
In this section, we calculate the conductor zeta function using Poisson summation. The function $f_W / \Phi^s$ is a product of local functions $f_v/ \Phi_v^s$ on $\operatorname{Hom}\left(k_v^*, G\right)$. We define its global Fourier transform to be
$$
\widehat{f}_{W, G}(x ; s)=\int_{\chi \in \operatorname{Hom}\left(\mathbf{A}^*, G\right)} \frac{f_{W}(\chi)\langle\chi, x\rangle}{\Phi(\chi)^s} \mathrm{~d} \chi,
$$
where $x=\left(x_v\right)_v \in \mathbf{A}^* \otimes G^{\wedge}$.  We use $\widehat{f}_{W}$ for shorthand when there is no ambiguity.

Recall from Proposition \ref{absconvergence:delta}, we proved that the local Fourier transforms converge absolutely for $\operatorname{Re}(s)  \geq a(G) - \frac{1}{\delta p^e}$. By Lemma \ref{asympoticofTriviallocalft}, we conclude that
\begin{equation}\label{globalproductlocal}
    \widehat{f}_{W, G}(x ; s)=\prod_v \widehat{f}_{W_v, G}\left(x_v ; s\right), \quad \text{ for } \operatorname{Re}(s)  \gg 1. 
\end{equation}
Later in this section, we will make a continuation of the Euler product (\ref{globalproductlocal}) using Dedekind zeta functions.
\subsection{Poisson summation}
Now we prove a version of Poisson summation. This is the function field version of \cite[Prop.~3.9]{Frei_2022}. Recall that $S  \subseteq \Omega_k$ is a finite set of places where local conditions are imposed. To prove Theorem \ref{maintheorem}, we may enlarge $S$ by finitely many places. Therefore, in order to apply Poisson summation, we enlarge $S$ such that $\mathcal{O}_S$ has trivial class group.

 We view $\mathcal{O}_S^* \otimes G^{\wedge}$ as a subgroup of $k^* \otimes G^{\wedge}$ in the following way: we have the exact sequence
\begin{equation}\label{fractionalidealses} 0 \rightarrow \mathcal{O}_S^* \rightarrow k^* \rightarrow P\left(\mathcal{O}_S\right) \rightarrow 0,
\end{equation}
where $P\left(\mathcal{O}_S\right)$ denotes the group of non-zero principal fractional ideals of $\mathcal{O}_S$. Since $P\left(\mathcal{O}_S\right)$ is a free abelian group, we have $\operatorname{Tor}\left(P\left(\mathcal{O}_S\right), G^{\wedge}\right)=0$. Therefore, applying $(\cdot) \otimes G^{\wedge}$ to (\ref{fractionalidealses}) we find that the map $\mathcal{O}_S^* \otimes G^{\wedge} \rightarrow k^* \otimes G^{\wedge}$ is injective, as required.
\begin{proposition}\label{PossionS:the prop}
For $\operatorname{Re} s  \gg 1$, the Fourier transform $\widehat{f}_{W, G}(\cdot ; s)$ defines a holomorphic function on this domain. Moreover, if $S \subseteq \Omega_k$ is finite such that $\mathcal{O}_S$ has trivial class group we have the Poisson formula
\begin{equation}\label{Possionsummation}
\sum_{\chi \in \operatorname{Hom}\left(\mathbf{A}^* / k^*, G\right)} \frac{f_{W}(\chi)}{\Phi(\chi)^s}=\frac{|G|}{\left|\F^{^*}_{q} \otimes G^{\wedge}\right|} \sum_{x \in \mathcal{O}_S^* \otimes G^{\wedge}} \widehat{f}_{W, G}(x ; s), \quad \operatorname{Re} s \gg 1 .
\end{equation}
Note that the group $\mathcal{O}_S^* \otimes G^{\wedge}$ is finite by Dirichlet's $S$-unit theorem; in particular, the right-hand sum is finite.
\end{proposition}
\begin{proof}
Let $x \in \mathcal{O}_S^* \otimes G^{\wedge}$. Let $x_v$ denote its image in $k_v^* \otimes G^{\wedge}$. Recall that we have normalised our Haar measures on $\operatorname{Hom}\left(k_v^*, G\right)$ to be $|G|^{-1}$ times the counting measure.
    
Now Lemma \ref{computelocalinOv} and (\ref{globalproductlocal}) gives
\begin{align*}
\widehat{f}_{W, G}(x ; s) & =\frac{1}{|G|^{\left|S\right|}} \prod_{v \in S} \sum_{\chi_v \in \operatorname{Hom}\left(k_v^*, G\right)} \frac{f_{W_v}(\chi_v)\left\langle\chi_v, x_v\right\rangle}{\Phi_v\left(\chi_v\right)^s} \prod_{v \notin S} \sum_{\chi_v \in \operatorname{Hom}\left(\mathcal{O}_v^*, G\right)} \frac{f_{W_v}(\chi_v)\left\langle\chi_v, x_v\right\rangle}{\Phi_v\left(\chi_v\right)^s} \\
& =\frac{1}{|G|^{\left|S\right|}} \sum_{\chi \in \operatorname{Hom}\left(\mathbf{A}_S^*, G\right)} \frac{f_{W}(\chi)\langle\chi, x\rangle}{\Phi(\chi)^s},
\end{align*}
where $\mathbf{A}_S^*=\prod_{v \in S} k_v^* \times \prod_{v \notin S} \mathcal{O}_v^*$. Since $\mathcal{O}_S^* \otimes G^{\wedge} $  is finite and the sum over $\chi$ is absolutely convergent (\ref{absconvergence:delta}), changing the order of summation yields 
$$
\sum_{x \in \mathcal{O}_S^* \otimes G^{\wedge}} \widehat{f}_{W, G}(x ; s)=\frac{1}{|G|^{\left|S\right|}} \sum_{\chi \in \operatorname{Hom}\left(\mathbf{A}_S^*, G\right)} \frac{f_{W}(\chi)}{\Phi(\chi)^s} \sum_{x \in \mathcal{O}_S^* \otimes G^{\wedge}}\langle\chi, x\rangle .
$$

As $\mathbf{A}_S^*$ and $\mathbf{A}_S^* / \mathcal{O}_S^*$ are locally compact groups and their subgroups of $n$-th powers are closed by Lemma \ref{TopolocyofA}, an application of \cite[Lem.~3.2]{MR3884640}, gives canonical isomorphisms of abelian groups $\operatorname{Hom}\left(\mathbf{A}_S^*, G\right) \cong\left(\mathbf{A}_S^* \otimes G^{\wedge}\right)^{\wedge}$ and $\operatorname{Hom}\left(\mathbf{A}_S^* / \mathcal{O}_S^*, G\right) \cong$ $\left(\mathbf{A}_S^* / \mathcal{O}_S^* \otimes G^{\wedge}\right)^{\wedge}$. Therefore, we can view an element $\chi \in \operatorname{Hom}\left(\mathbf{A}_S^*, G\right)$ as a character of $\mathbf{A}_S^* \otimes G^{\wedge}$. It is easily seen that $\chi$ induces the trivial character on $\mathcal{O}_S^* \otimes G^{\wedge}$ if and only if $\chi \in \operatorname{Hom}\left(\mathbf{A}_S^* / \mathcal{O}_S^*, G\right)$. Thus, we may apply character orthogonality to find that
$$
\sum_{x \in \mathcal{O}_S^* \otimes G^{\wedge}}\langle\chi, x\rangle= \begin{cases}\left|\mathcal{O}_S^* \otimes G^{\wedge}\right|, & \text { if } \chi \in \operatorname{Hom}\left(\mathbf{A}_S^* / \mathcal{O}_S^*, G\right), \\ 0, & \text { otherwise }.\end{cases}
$$
We therefore obtain
$$
\sum_{x \in \mathcal{O}_S^* \otimes G^{\wedge}} \widehat{f}_{W, G}(x ; s)=\frac{\left|\mathcal{O}_S^* \otimes G^{\wedge}\right|}{|G|^{\left|S\right|}} \sum_{\chi \in \operatorname{Hom}\left(\mathbf{A}_S^* / \mathcal{O}_S^*, G\right)} \frac{f_{W}(\chi)}{\Phi(\chi)^s}.
$$
Dirichlet's $S$-unit theorem gives an  isomorphism $\mathcal{O}_S^* \cong \F_q^* \times \mathbb{Z}^{|S| - 1}$, therefore,
$$
\frac{\left|\mathcal{O}_S^* \otimes G^{\wedge}\right|}{|G|^{\left|S\right|}}=\frac{\left|\F_q^*   \otimes G^{\wedge}\right|}{|G|}.
$$
Moreover, as $\mathcal{O}_S$ has trivial class group, the natural map $\mathbf{A}_S^* / \mathcal{O}_S^* \rightarrow \mathbf{A}^* / k^*$ is an isomorphism \cite[Lem.~2.8]{MR2581243}.
\end{proof}
\subsection{Analytic continuation of the global Fourier transforms} 

From Proposition \ref{PossionS:the prop} and (\ref{globalproductlocal}), we have that
\begin{equation}\label{Possionsummationwithlocal}
\sum_{\chi \in \operatorname{Hom}\left(\mathbf{A}^* / k^*, G\right)} \frac{f_{W}(\chi)}{\Phi(\chi)^s}=\frac{|G|}{\left|\F_q^* \otimes G^{\wedge}\right|} \sum_{x \in \mathcal{O}_S^* \otimes G^{\wedge}} \prod_v \widehat{f}_{W_v, G}\left(x_v ; s\right), \quad \operatorname{Re} s>M .
\end{equation}

 We define the following notation to make the rest of our calculation easier to read.
 Let $a \leq b$ and let $g(s)$ and $f(s)$ be holomorphic functions on the half-plane $ \Real s>b$. We write $f(s)\approx_a g(s)$
if there exists a holomorphic function $t(s)$ on $\Real(s) > a$ such that $f(s) = t(s)g(s)$ for $\Real(s) > b$.
\begin{proposition}
    Let $p\nmid l$ be a positive integer. Consider the constant extension $L = k(\mu_l)/k$ where $\mu_l$ denotes the $l$-th roots of unity.
    Then the Dedekind zeta function of $L$ satisfies
\begin{equation}\label{dedekindofL}
        \zeta_L(s) \approx_\frac{1}{2} \prod_{\substack{v\in \Omega_k \\ \text{ splits completely}}}\left(1+\frac{[L:k]}{q_v^s}\right),
    \end{equation}
    where the products runs over all the places  $v \in \Omega_k$ such that $v$ splits completely in the field extension $L/k$. 
\end{proposition}
\textit{Proof.}
Since $L/k$ is a constant extension, it is everywhere unramified \cite[Prop.~8.5]{Rosen_2002}, and a place splits completely if and only if its Frobenius is trivial. Recall, 
$ \zeta_L(s):= \prod_{\omega\in\Omega_{L}} \left(1-q^{-s\deg \omega}\right)^{-1}  = \prod_{v\in \Omega_k}\prod_{\substack{w\in\Omega_L\\w\mid v}}\left(1-\frac{1}{q_w^s}\right)^{-1}$.
As $L/k$ is Galois, we have
\begin{align*}
   \prod_{v\in \Omega_k}\left(1-\frac{1}{q_v^{f_vs}}\right)^{-|\{w \in \Omega_{k(\mu_\alpha)} :  w \mid v\}|}
    &\approx_\frac{1}{2} \prod_{v\in \Omega_k}\left(1+\frac{|\{w \in \Omega_{k(\mu_\alpha)} : w\mid v\}|}{q_v^{f_vs}}\right)
    \\ &\approx_\frac{1}{2} \prod_{\substack{v\in \Omega_k \\ \text{ splits completely}}}\left(1+\frac{[L:k]}{q_v^s}\right). 
    \quad \quad\quad\quad\quad\quad \quad \quad\qed
\end{align*}
We now calculate the global Fourier transform at $x = \mathbf{1}$.
\begin{proposition}\label{prop_trivial_char_decomp}
There exists $\epsilon >0$ and a function $D(k,G,W;s)$ which is holomorphic in the region $\Real(s) \geq a(G) - \epsilon,$ such that
\begin{equation*}
    \prod_v \widehat{f}_{W_v, G}\left(\mathbf{1}_v ; s\right) = D(k,G,W;s)\Theta(k,G,W;s),
\end{equation*}
where
\begin{align*}
\Theta(k,G,W;s):=&\zeta_k(s)^{-1}\prod_{\alpha \mid |G_t|}\zeta_{k(\mu_\alpha)}(s)^{d(k,\alpha)}\\
&\times\prod^{\infty}_{\substack{j=2\\v_p(j-1)\leq e-1}}\prod_{\alpha\,\mid \, |G_t|}\zeta_{k(\mu_\alpha)}\left(js- \sum_{m=1}^er_m(G)\left(j-1-\left\lfloor\frac{j-1}{p^m}\right\rfloor\right)\right)^{d(k,\alpha)} \\
&\times\prod^{\infty}_{\substack{j=2\\v_p(j-1)\leq e-1}}\prod_{\alpha\,\mid \, |G_t|}\zeta_{k(\mu_\alpha)}\left(js- \sum_{m=1}^er_m(G)\left(j-1-\left\lceil\frac{j-1}{p^m}\right\rceil\right)\right)^{-d(k,\alpha)}.
\end{align*}
\end{proposition}
\textit{Proof.}
Since $x = \mathbf{1} $ is trivial, Lemma \ref{LocalFT} implies that
    \begin{align*}
        \prod_v \widehat{f}_{ v} \left(\mathbf{1}_v ; s\right)
        =&\prod_v \bigg(1 + \frac{|\Hom (\F_{q_v}^*, G_t)| - 1 }{q_v^{s}} + |\Hom (\F_{q_v}^*, G_t)| \\&\times\sum^{\infty}_{\substack{j=2\\v_p(j-1)\leq e-1}} \frac{\left(q_v^{\sum^e_{m=1+v_p(j-1)}r_m(G)}-1\right)\left(q_v^{\sum_{m=1}^er_m(G)\left(j-1-\left\lceil\frac{j-1}{p^m}\right\rceil\right)}\vphantom{\prod_v} \right)}{q_v^{js}}\bigg)  \\&\approx_{a(G)-\epsilon}\prod_v\left(1+ \frac{|\Hom (\F_{q_v}^*, G_t)|  }{q_v^{s}}\right)\prod_v\left(1- \frac{ 1 }{q_v^{s}}\right)\\&\times \prod^{\infty}_{\substack{j=2\\v_p(j-1)\leq e-1}} \prod_v\left(1+ \frac{|\Hom (\F_{q_v}^*, G_t)|  }{q_v^{js- \sum_{m=1}^er_m(G)\left(j- 1- \left\lfloor\frac{j-1}{p^m}\right\rfloor\right)}}\right)
        \\ &\times\prod^{\infty}_{\substack{j=2\\v_p(j-1)\leq e-1}} \prod_v\left(1- \frac{|\Hom (\F_{q_v}^*, G_t)|  }{q_v^{js- \sum_{m=1}^er_m(G)\left(j-1-\left\lceil\frac{j-1}{p^m}\right\rceil\right)}}\right).
\end{align*}

where the last equality relies on Lemma \ref{asympoticofTriviallocalft}. We also observe that after applying the change of variables, it suffices to deal with the analytical behavior of the Euler product 
$\prod_v\left(1+ \frac{|\Hom (\F_{q_v}^*, G_t)|  }{q_v^{s}}\right)$.  

Since $\Bbb F_{q_v}^*$ is cyclic, a homomorphism $F_{q_v}^\times\to G_t$
is determined by $g\in G_t$ with $g^{q_v-1}=1$, i.e. $\operatorname{ord}(g)\mid(q_v-1)$.
Each cyclic subgroup of order $\alpha$ contributes $\phi(\alpha)$ such elements,
and there are $r(\alpha)$ of them. Hence
\[
|\mathrm{Hom}(\Bbb F_{q_v}^\times,G_t)|
=\sum_{\alpha\mid|G_t|}r(\alpha)\varphi(\alpha)\mathbf1_{\{q_v\equiv1\pmod\alpha\}}.
\]
Therefore,
\begin{align*}
    \prod_v\left(1+ \frac{|\Hom (\F_{q_v}^*, G_t)|  }{q_v^{s}}\right)  \approx_{\frac{1}{2} }\prod_{\alpha \mid |G_t|}\prod_{\substack{v\in \Omega_k\\q_v \equiv 1 \bmod{\alpha} }}\left(1+ \frac{r(\alpha)\phi(\alpha)}{q_v^{s}}\right).
\end{align*}
For $\alpha \mid |G_t|$, consider the constant extension $ k(\mu_\alpha)/k$. We have that $v \in \Omega_k$ splits completely if and only if $q_v \equiv 1 \bmod{\alpha}$, as $\operatorname{Frob}_v(\mu_\alpha) = \mu_\alpha^{q_v}$.
Finally, by (\ref{dedekindofL}), we have that
\begin{align*}
     \prod_{\alpha \mid |G_t|}\prod_{\substack{v\in \Omega_k\\q_v \equiv 1 \bmod{\alpha} }}\left(1+ \frac{r(\alpha)\phi(\alpha)}{q_v^{s}}\right) & \approx_\frac{1}{2} \prod_{\alpha \mid |G_t|} \prod_{\substack{v\in\Omega_k\\ \text{ splits completely }}}\left(1+\frac{[k(\mu_\alpha):k]}{q_v^s}\right)^{d(k,\alpha)}\\& \approx_{\frac{1}{2}} \prod_{\alpha \mid |G_t|} \zeta_{k(\mu_\alpha)}(s)^{d(k,\alpha)}.\qquad\qquad\qquad\qquad\qquad\qquad\qquad\qed
\end{align*} 
Let $x = (x_t,x_w) \in \mathcal{O}_S^* \otimes (G_t\times G_w)^{\wedge}$ be a non-trivial character, we analyze its global Fourier transform $\widehat{f}_{W,G}(x,s)$ using the theory built in \S2.  The following results will help us to determine the shape of the Euler product of $\widehat{f}_{W,G}(x,s)$.

\begin{lemma}\label{lem: generalized-ABdul-Pth-valuation}
Let $x \in k^{*p^{t}}$ with $x \notin k^{*p^{t+1}}$. Then the $v$-adic expansion of $x$ in $k_v$ has the form
\[
x = a_{0,v}+a_{p^t,v}\pi_v^{p^t} +\cdots,
\]
where $a_{p^t,v}$ is non-zero for all but finitely many places $v\in \Omega_k$.
\end{lemma}

\begin{proof}
The case $t=0$ is \cite[Coro.~6.15]{alfaraj2025maninsconjectureequivariantcompactifications}. For $t \geq 1$, write $x = y^{p^t}$ with $y \notin k^{*p}$. Applying \cite[Coro.~6.15]{alfaraj2025maninsconjectureequivariantcompactifications} to $y$ gives the desired non-vanishing for all but finitely many places.
\end{proof}
\begin{proposition}\label{Prop:hom_is_frob}
For $p \nmid |G_t|$, the function
\begin{equation}\label{thehomfrob}
  \rho_x(v) =  \begin{cases}
    |\Hom(\F_{q_v}^*,G_t)|,& x_v \in k_v^{*|G_t|}, \\0,&\text{otherwise}.
    \end{cases}
\end{equation}
is $S$-frobenian.
\end{proposition}
\begin{proof}
As we are in the tame case, the proof is analogous to the number field case, which is a special case in \cite[Lem.~3.13]{Frei_2022} by taking $\mathcal{A}= \{1\}$. 
\end{proof}

\begin{lemma}\label{non-trivial_char_ana_cont}
     For $x \in \mathcal{O}_S^*\otimes G^{\wedge}$, there exists a constant extension $M_x/k$ such that $\widehat{f}_{W,G}(x,s) $ equals
     \begin{equation}\label{Decomp_of_non_tri_char}
          P(x,s)\zeta_k(s)^{-1}\zeta_{M,x}(s)^{m(\rho_x)}\times \prod^{h}_{\substack{j=2\\v_p(j-1)\leq e-1}} \zeta_{M,x}\left({js- \sum_{m=1}^er_m(G)\left(j- 1- \left\lfloor\frac{j-1}{p^m}\right\rfloor\right)}\right)^{m(\rho_x)} ,
     \end{equation}
     where $2 \leq h \leq p^e$, and $P(x,s)$  admits analytic continuation for $\Real s \geq 1$, except for possible algebraic singularities along $\Real (s) =1 $.
Moreover, the real part of the order of these singularities of $P(x,s)$   on $\Real (s) =1 $ is strictly less than $m(\rho)$.
\end{lemma}
\begin{proof}
Let $x = (x_t,x_w) \in \mathcal{O}_S^* \otimes (G_t\times G_w)^{\wedge}$. By Lemma~\ref{lem: generalized-ABdul-Pth-valuation}, there exist $2\leq h \leq p^e$ such that
\begin{align*}
 \widehat{f}(x;s) &\approx_{a(G)}\prod_{v:x_{t,v} \in k_v^{*|G_t|}}\left(1-\frac{1}{q_v^s}\right)\prod_{v:x_{t,v} \notin k_v^{*|G_t|}}\left(1+ \frac{|\Hom (\F_{q_v}^*, G_t)|  }{q_v^{s}}\right)\prod_{v:x_{t,v} \in k_v^{*|G_t|}}\left(1- \frac{ 1 }{q_v^{s}}\right)\\&\times \prod^{h}_{\substack{j=2\\v_p(j-1)\leq e-1}} \prod_{v:x_{t,v} \in k_v^{*|G_t|}}\left(1+ \frac{|\Hom (\F_{q_v}^*, G_t)|  }{q_v^{js- \sum_{m=1}^er_m(G)\left(j- 1- \left\lfloor\frac{j-1}{p^m}\right\rfloor\right)}}\right)
 \\&\approx_{a(G)}\zeta_k(s)^{-1}\prod_{v:x_{t,v} \in k_v^{*|G_t|}}\left(1+ \frac{|\Hom (\F_{q_v}^*, G_t)|  }{q_v^{s}}\right)\\&\times \prod^{h}_{\substack{j=2\\v_p(j-1)\leq e-1}} \prod_{v:x_{t,v} \in k_v^{*|G_t|}}\left(1+ \frac{|\Hom (\F_{q_v}^*, G_t)|  }{q_v^{js- \sum_{m=1}^er_m(G)\left(j- 1- \left\lfloor\frac{j-1}{p^m}\right\rfloor\right)}}\right).
\end{align*}

Applying Propsition~\ref{Prop:Frob_func_into_zeta} yields
\begin{equation}\label{one-time}
     \prod_{v:x_{t,v} \in k_v^{*|G_t|}}\left(1+ \frac{|\Hom (\F_{q_v}^*, G_t)|  }{q_v^{s}}\right) = P(x,s)\zeta_{M_x}(s)^{m(\rho_x)},
\end{equation}
as the product above is associated to the frobenian function \eqref{thehomfrob}. 
The assertion now follows by applying the change of variables to \eqref{one-time}.
\end{proof}
\subsection{The asymptotic formula}
So far, we have related the conductor series of $G-$extensions to global Fourier transforms via Proposition \ref{PossionS:the prop}. We are going to apply Tauberian theorem to our Dirichlet series, therefore we need to investigate which terms in \eqref{removesubgroups} are contribute towards the leading term in the asymptotic formula.

\begin{proposition}\label{which_cha_contri}
Let $x \in \mathcal{O}_{S}^*\otimes G^{\wedge}$. Denote by $\omega(k,G,x)$ the order of the leading singularity of $\widehat{f}_{W,G}(x,s)$ at $\Real(s) = a(G)$, then 
\begin{enumerate}
    \item $\omega(k,G,x) < \omega (k,G,\mathbf{1})$ for $x$ not trivial in $ \mathcal{O}_{S}^* \otimes G^{\wedge}.$
     \item Let $H\lneq G$ be a subgroup. Then either $a(H)<a(G)$, or $a(H)=a(G)$ and every singularity of $\widetilde F_{H,W}(s)$ on the line $\Real(s)=a(G)$ has order strictly smaller than $b(k,G)$.
\end{enumerate}
\end{proposition}
\begin{proof}
According to the definition of the frobenian function \eqref{thehomfrob}, we have $m(\rho_{x_t}) \leq m(\rho_{\mathbf{1}_t})$, with equality if and only if $x_{t,v}$ is a $|G_t|$-th power for all but finitely many places. By Grunwald-Wang \cite[Thm.~9.1.11]{CohoNF}, this happens if and only if $x_t = \mathbf{1}_t.$ Therefore, $\omega(k,G,\mathbf{1}) > \omega(k,G,x) $, for $x_t$ non-trivial. 

We now prove \textit{(2)}. For $H=1$ the series $\widetilde F_{1,W}(s)$ is constant and there is nothing to prove, so assume $H\neq 1$. Write $H=H_w\times H_t$ with $H_w\leq G_w$ and $H_t\leq G_t$. For a finite abelian $p$-group $A_w$, write $A_w\cong C_{p^{e_1(A)}}\times\cdots\times C_{p^{e_{R_A}}(A)}$ with $e_1(A)\geq\cdots\geq e_{R_A}(A)\geq 1$, so that $e_1(A)$ equals the invariant $e$ of $A$; when $A_w$ is trivial we set $R_A=0$ and $e_1(A)=0$. Set $R=R_G$, $r=R_H$ and $e_i(H)=0$ for $r<i\leq R$. For every $t\geq 1$, the number of indices $i$ with $e_i(H)\geq t$ equals $\dim_{\F_p}\left(p^{t-1}H_w\right)[p]$, and $p^{t-1}H_w\subseteq p^{t-1}G_w$ gives $\dim_{\F_p}\left(p^{t-1}H_w\right)[p]\leq\dim_{\F_p}\left(p^{t-1}G_w\right)[p]$; hence $e_i(H)\leq e_i(G)$ for all $i$.

From the definition of $a(\cdot)$, if $A_w$ is non-trivial, $a(A)=p^{-e_1(A)}+\sum_{i=1}^{R_A}\left(1-p^{-e_i(A)}\right)=R_A-\sum_{i=2}^{R_A}p^{-e_i(A)}$, while $a(A)=1$ when $A_w$ is trivial. If $G_w$ and $H_w$ are non-trivial, then
\begin{equation*}
a(G)-a(H)=\sum_{i=r+1}^{R}\left(1-p^{-e_i(G)}\right)+\sum_{i=2}^{r}\left(p^{-e_i(H)}-p^{-e_i(G)}\right),
\end{equation*}
and every summand is non-negative, since $e_i(G)\geq 1$ in the first sum and $e_i(H)\leq e_i(G)$ in the second; as the summands of the first sum are positive, the equality $a(H)=a(G)$ forces $r=R$ and $e_i(H)=e_i(G)$ for $2\leq i\leq R$. If $G_w$ is cyclic, then $H_w$ is cyclic and $a(H)=a(G)=1$. If $G_w$ is not cyclic and $H_w$ is cyclic or trivial, then $a(H)=1<a(G)$ by Lemma~\ref{LocationofPoleComputation} \textit{(4)}, and \textit{(2)} holds in this case.

Suppose from now on that $a(H)=a(G)$. If $\psi_v\notin\Hom(\Gamma_{k_v},H)$ for some $v\in S$, then $\widehat f_{W,H}(\mathbf 1,s)$ vanishes identically and there is nothing to prove; assume otherwise. We show $b(k,H)<b(k,G)$; by Lemma~\ref{structure of Lambda} applied to the finite abelian group $H$, this gives $\omega(k,H,\mathbf 1)\leq b(k,H)<b(k,G)=\omega(k,G,\mathbf 1)$, as required. We use two observations. First, $\sum_{\alpha\mid|A_t|}d_A(k,\alpha)=\sum_{y\in A_t}[k(\mu_{\operatorname{ord}(y)}):k]^{-1}$, where $d_A(k,\alpha)$ denotes the quantity $d(k,\alpha)$ formed with $A_t$ in place of $G_t$; this sum is at least $1$, and it increases strictly when $A_t$ is replaced by a group properly containing it. Second, for non-cyclic $A_w$ and $f'\geq 0$, the group $p^{f'}A_w$ is cyclic if and only if at most one of the $e_i(A)$ exceeds $f'$, if and only if $f'\geq e_2(A)$; hence the invariant $f$ of $A$ equals $e_2(A)$, and the difference $e-f$ for $A$ equals $e_1(A)-e_2(A)$.

If $G_w$ is cyclic, then $b(k,G)=p^{e_1(G)}\left(\sum_{\alpha}d(k,\alpha)\right)-1$ and $b(k,H)=p^{e_1(H)}\left(\sum_{\alpha}d_H(k,\alpha)\right)-1$ by \eqref{bG} applied to $G$ and to $H$. Either $e_1(H)<e_1(G)$, and then $p^{e_1(H)}\sum_{\alpha}d_H(k,\alpha)\leq p^{e_1(G)-1}\sum_{\alpha}d(k,\alpha)<p^{e_1(G)}\sum_{\alpha}d(k,\alpha)$; or $e_1(H)=e_1(G)$, and then $H_w=G_w$, as $H_w$ is a subgroup of the cyclic group $G_w$ of the same order, so $H_t\lneq G_t$ and the first observation gives $\sum_{\alpha}d_H(k,\alpha)<\sum_{\alpha}d(k,\alpha)$. In both cases $b(k,H)<b(k,G)$.

If $G_w$ is not cyclic, then $a(H)=a(G)>1$ forces $H_w$ to be non-cyclic by Lemma~\ref{LocationofPoleComputation} \textit{(4)} applied to $H$, and the equality case above gives $r=R$, $e_i(H)=e_i(G)$ for $i\geq 2$ and $e_1(H)\leq e_1(G)$. By the second observation, $b(k,G)=p^{e_1(G)-e_2(G)}\sum_{\alpha}d(k,\alpha)$ and $b(k,H)=p^{e_1(H)-e_2(H)}\sum_{\alpha}d_H(k,\alpha)$, with $e_2(H)=e_2(G)$. If $e_1(H)<e_1(G)$, then $b(k,H)\leq p^{e_1(G)-e_2(G)-1}\sum_{\alpha}d(k,\alpha)<b(k,G)$. If $e_1(H)=e_1(G)$, then $e_i(H)=e_i(G)$ for all $i$, so $H_w=G_w$, being a subgroup of the same order; hence $H_t\lneq G_t$ and the first observation gives $b(k,H)<b(k,G)$.
\end{proof}

In the asymptotic formula, the order of the singularity appears in the exponent of the logarithmic factor, but only the real part of the order matters for the growth rate. Therefore, by Proposition \ref{which_cha_contri}, only the trivial character will contribute towards the leading term. Henceforth, we make a detailed analysis for the trivial character here. We require the following result from \cite[Prop.~6.4]{Lagemann_2015}.
\begin{lemma}\label{lagemanncalcualtion}
    For any integer $2\leq j \leq p^e$, let
    \begin{equation*}
        \epsilon(j) = a(G) - \frac{1+ \sum_{m=1}^er_m(G)\left(j-1-\left\lfloor\frac{j-1}{p^m}\right\rfloor\right)}{j},
    \end{equation*}
    and let $0 \leq f \leq e$ be the minimal integer such that $p^fG_w$ is cyclic. Then we have $\epsilon(j) \geq 0$ with $\epsilon(j) = 0$ if and only if $p^f \mid j$.
\end{lemma}
\begin{proof}
    The exact statement from Lagemann's paper looks slightly different from here, as he was working with the $p^m-$rank of $G_w$, defined as  $R_m(G):= \log_p(p^{m-1}G_w:p^{m}G_w)$. These two statements are identical by observing that
\begin{align*}
&\sum_{m=1}^e\left(R_m(G)-R_{m+1}(G)\right)\left(j-1-\left\lfloor\frac{j-1}{p^m}\right\rfloor\right)
 =\sum_{m=1}^e\left( \left\lfloor\frac{j-1}{p^{m-1}}\right\rfloor - \left\lfloor\frac{j-1}{p^{m}}\right\rfloor\right)R_m(G),
\end{align*}
since $R_{e+1}(G) = 0$.
\end{proof}
\begin{lemma}\label{structure of Lambda}
Recall from Theorem~\ref{easyED} $a(G) = \frac{1+ \sum_{m=1}^er_m(G)(p^e- p^{e-m})}{p^e} = z_{p^e}(G)$, and the definition of $z_j(G)$ and $y_j(G)$ from Lemma \ref{LocationofPoleComputation}. Then $\Theta(k,G,W;s)$ from Proposition~\ref{prop_trivial_char_decomp} is absolutely convergent for $\Real(s) > a(G)$, and there exists $\epsilon > 0$ such that $\Theta(k,G,W;s)$ extends to a meromorphic function on $\Real(s) > a(G) - \epsilon$ with the following properties:
    \begin{enumerate}
        \item There exist n$ \in$ $\Z_{\geq 0}$ and real
    numbers  $0=\theta_0<\theta_1<\ldots<\theta_n<1$ such that, for $h \in\{0, \ldots, n\}$ and $m\in\Z$, the function $\Theta(k,G,W;s)$ has a pole of order $b_h\leq b(k,G)$ at $s_h+\frac{2 \pi i m}{\log q}$, where $s_h = a(G)+\frac{2 \pi i \theta_h}{\log q}$ and $b(k,G)$ is as in Theorem \ref{maintheorem}.
        \item $\Theta(k,G,W;s)$ is holomorphic for $$\Real(s)>a(G)-\epsilon, s \notin\left\{s_h+\frac{2 \pi i m}{\log q}: h \in\{0, \ldots, n\}, m \in \Z\right\}.$$
        \item Recall $D$ as defined in (\ref{fundamentalRegionD(shift this)}). The set $B(k,G):=\{s_h \in D: b_h = b(k,G)\}$ contains $s_0 =a(G)$ and is evenly spread: there exists $\gamma \in \Z_{\geq 1}$ such that $B(k,G) = \{a(G)+\frac{2 t \pi i }{\gamma \log q}: t\in\{0,\ldots,\gamma-1\}\}$.
        \end{enumerate}
\end{lemma}
\begin{proof}
Let $j\geq 2$ and let $\alpha\mid|G_t|$ with $r(\alpha)>0$. By Lemma~\ref{lemma:value_zeta_function}, the poles from the Dedekind zeta function $\zeta_{k(\mu_\alpha)}\left(js-\sum_{m=1}^er_m(G)\left(j-1-\left\lfloor\frac{j-1}{p^m}\right\rfloor\right)\right)$ are located at $z_j(G)+\frac{2\pi i n}{j[k(\mu_\alpha):k]\log q}$ for $n\in\Z$, and are simple; its zeros are located on the vertical line of real part $z_j(G)-\frac{1}{2j}$, by the Riemann hypothesis for function fields. Now we observe that by Lemma~\ref{LocationofPoleComputation} \textit{(3)}, we only need to consider these zeta functions with $j\leq p^e$; note that every $2\leq j\leq p^e$ satisfies $v_p(j-1)\leq e-1$, so all these zeta functions occur in $\Theta(k,G,W;s)$.

Furthermore, we have that the poles from $\zeta_{k(\mu_\alpha)}\left(js-\sum_{m=1}^er_m(G)\left(j-1-\left\lceil\frac{j-1}{p^m}\right\rceil\right)\right)$ are located at $y_j(G)+\frac{2\pi i n}{j[k(\mu_\alpha):k]\log q}$ for $n\in\Z$, with zeros on the vertical line of real part $y_j(G)-\frac{1}{2j}$. As $y_j(G)\leq z_j(G)$ by Lemma~\ref{LocationofPoleComputation} \textit{(1)}, we can focus on these zeta functions with $j\leq p^e$ as well. Since $d(k,1)=1$, the factor $\zeta_k(s)^{d(k,1)}$ in the first product of $\Theta(k,G,W;s)$ cancels the factor $\zeta_k(s)^{-1}$; the remaining tame factors $\zeta_{k(\mu_\alpha)}(s)^{d(k,\alpha)}$, $\alpha\neq 1$, have simple poles at $1+\frac{2\pi in}{[k(\mu_\alpha):k]\log q}$ and zeros on the vertical line of real part $\frac{1}{2}$.

If $G_w$ is cyclic, we have that the pole with the largest real part among these zeta factors is located at $s = a(G)=1$, by Lemma~\ref{LocationofPoleComputation} \textit{(2)} and \textit{(4)}. Note that in this case, the poles from the tame factors $\zeta_{k(\mu_\alpha)}(s)$ are also located along $\Real(s)=a(G)$. If $G_w$ is not cyclic, we have $a(G)>1$ by Lemma~\ref{LocationofPoleComputation} \textit{(4)}, so the tame factors do not reach the line $\Real(s)=a(G)$.

To figure out which zeta functions have their poles with the largest real part, we apply Lemma~\ref{lagemanncalcualtion}. We conclude that the factors that contribute towards the order of the poles with the real part $a(G)$ are $\zeta_{k(\mu_\alpha)}\left(js-\sum_{m=1}^er_m(G)\left(j-1-\left\lfloor\frac{j-1}{p^m}\right\rfloor\right)\right)$ with $p^f\mid j$, together with, when $G_w$ is cyclic, the tame factors $\zeta_{k(\mu_\alpha)}(s)$ with $\alpha\neq 1$, and there is no cancellation between zeros and poles: by Lemma~\ref{LocationofPoleComputation} \textit{(2)}, the poles of the factors with exponent $-d(k,\alpha)$ satisfy $y_j(G)<z_j(G)=a(G)$ for $p^f\mid j$ and $y_j(G)\leq z_j(G)<a(G)$ otherwise, and the zeros of all factors lie on vertical lines of real part strictly smaller than $a(G)$ by the locations recorded above. Parts \textit{(1)} and \textit{(2)} follow for $\epsilon$ small enough.

We now prove \textit{(3)}. Regard the tame factor $\zeta_{k(\mu_\alpha)}(s)^{d(k,\alpha)}$ as the case $j=1$ of the contributing factors, so that these are indexed by the pairs $(j,\alpha)$ with $2\leq j\leq p^e$ and $p^f\mid j$, together with the pairs $(1,\alpha)$ with $\alpha\neq 1$ when $G_w$ is cyclic. Write $s_\theta=a(G)+\frac{2\pi i\theta}{\log q}$ with $\theta\in[0,1)$. The factor indexed by $(j,\alpha)$ has a pole at $s_\theta$ if and only if $j[k(\mu_\alpha):k]\,\theta\in\Z$, so the order of the pole of $\Theta(k,G,W;s)$ at $s_\theta$ equals the sum of $d(k,\alpha)$ over the pairs with $j[k(\mu_\alpha):k]\,\theta\in\Z$. For $\theta=0$ every pair contributes, and the order equals $\sum_{\alpha\neq 1}d(k,\alpha)+(p^e-1)\sum_{\alpha}d(k,\alpha)=p^e\sum_{\alpha}d(k,\alpha)-1=b(k,G)$ when $G_w$ is cyclic, using $d(k,1)=1$, and $p^{e-f}\sum_{\alpha}d(k,\alpha)=b(k,G)$ otherwise; since the summands are positive, the order at any $s_\theta$ is at most $b(k,G)$, with equality if and only if $j[k(\mu_\alpha):k]\,\theta\in\Z$ for every pair. Let $\gamma$ be the greatest common divisor of the integers $j[k(\mu_\alpha):k]$ over all pairs; then the latter holds if and only if $\gamma\theta\in\Z$, since $\gamma$ divides each $j[k(\mu_\alpha):k]$ and is an integral combination of them. This proves \textit{(3)}, and the bound $b_h\leq b(k,G)$ in \textit{(1)}.
\end{proof}
Recall that $B(k,G)$ denotes the set of poles of maximal order $b(k,G)$ in the region $D$. The preceding calculation yields the following results.
\begin{corollary}\label{infoaboutpoles}
According to the proof of Lemma~\ref{structure of Lambda}, we conclude that $\Theta(k,G,W;s)$ has one of the following properties:
\begin{enumerate}
    \item If $G_w$ is cyclic, there exists a function $D(k,G,W;s)$ and $\epsilon>0$ such that $D(k,G,W;s)$ is holomorphic for $\Real(s)>a(G)-\epsilon$ with
    \begin{equation*}
    \Theta(k,G,W;s) = D(k,G,W;s)\zeta_k(s)^{-1}\prod_{\alpha \mid |G_t|}\zeta_{k(\mu_\alpha)}(s)^{d(k,\alpha)}\prod^{p^e}_{j=2}\prod_{\alpha\,\mid \, |G_t|}\zeta_{k(\mu_\alpha)}\left(js-(j-1)\right)^{d(k,\alpha)}.
    \end{equation*}
 
    Furthermore, let $\gamma = \gcd\bigl(\{j : 2\leq j \leq p^e\}\cup\{[k(\mu_\alpha):k] : 1\neq\alpha \mid |G_t|,\ r(\alpha)>0\}\bigr)$; we have $B(k,G) = \{a(G)+\frac{2t\pi i}{\gamma\log q} : t\in\{0,\ldots,\gamma-1\}\}$. In particular, $|B(k,G)| = 1$ whenever $p^e >2$.
    \item If $G_w$ is not cyclic, let $1\leq f\leq e$ be the smallest integer such that $p^fG_w$ is cyclic, then there exists a function $D(k,G,W;s)$ and $\epsilon>0$, such that $D(k,G,W;s)$ is holomorphic for $\Real(s)>a(G)-\epsilon$ with
    \begin{equation*}
    \Theta(k,G,W;s) = D(k,G,W;s)\prod^{p^e}_{\substack{j=2\\p^f\mid j}}\prod_{\alpha\,\mid \, |G_t|}\zeta_{k(\mu_\alpha)}\left(js-\sum_{m=1}^er_m(G)\left(j-1-\left\lfloor\frac{j-1}{p^m}\right\rfloor\right)\right)^{d(k,\alpha)}.
    \end{equation*}
    Furthermore, let $\gamma = \gcd(j[k(\mu_\alpha):k] : 2\leq j\leq p^e,\ p^f\mid j,\ \alpha\mid|G_t|,\ r(\alpha)>0) = p^f$. We have $B(k,G) = \{a(G)+\frac{2t\pi i}{\gamma\log q} : t\in\{0,\ldots,\gamma-1\}\}$ with $|B(k,G)|=p^f$; for $f=e$ this value equals $p^e$.
\end{enumerate}
 This proves formula \eqref{Eq:formulaforB}, concerning the size of $B(k,G)$.
\end{corollary}

\begin{proof}
The factorizations restate the definition of $\Theta(k,G,W;s)$ from Proposition~\ref{prop_trivial_char_decomp}: when $G_w$ is cyclic the shift of the layer $j$ equals $j-1-\left\lfloor\frac{j-1}{p^e}\right\rfloor=j-1$ for $2\leq j\leq p^e$, and in both cases $D(k,G,W;s)$ collects the remaining factors, which is holomorphic for $\Real(s)>a(G)-\epsilon$ since, by the proof of Lemma \ref{structure of Lambda}, the poles of the remaining factors lie on vertical lines with real parts bounded away from $a(G)$. The description of $B(k,G)$ is Lemma \ref{structure of Lambda} \textit{(3)}, where $\gamma$ is the greatest common divisor of the integers $j[k(\mu_\alpha):k]$ over the contributing pairs. In case \textit{(1)}, this equals the displayed value: the pairs with $\alpha=1$ contribute the integers $j$, the pairs with $j=1$ contribute the integers $[k(\mu_\alpha):k]$ with $\alpha\neq 1$, and every remaining product $j[k(\mu_\alpha):k]$ is divisible by the greatest common divisor of these; when $p^e>2$ the first set contains $2$ and $3$, whence $|B(k,G)|=1$. In case \textit{(2)}, the pair $j=p^f$, $\alpha=1$ contributes $p^f$, and $p^f$ divides $j[k(\mu_\alpha):k]$ for every pair, whence $\gamma=p^f$. Finally, every $1\neq\alpha$ with $r(\alpha)>0$ has a prime divisor $\ell$ with $[k(\mu_\ell):k]$ dividing $[k(\mu_\alpha):k]$, and every prime $\ell$ dividing $|G_t|$ is the order of an element of $G_t$; hence in case \textit{(1)} the greatest common divisor is unchanged when the second set is replaced by $\{[k(\mu_\ell):k] : \ell \text{ prime},\ \ell\mid|G_t|\}$, and \eqref{Eq:formulaforB} follows.
\end{proof}

\section{Calculation of the leading constant}

Now we are in a position to apply a Tauberian theorem to  $F_G(s)$, the conductor series of $G$-extensions
\begin{equation}\label{Dirichlet_series_F}
	F_G(s) = \sum_{M=0}^\infty \frac{ N(k,G,W,q^M)}{q^{Ms}},
\end{equation}
where $ N(k,G,W,q^M) := {\#\left\{\varphi \in  G-\operatorname{ext}(k) : 
    \varphi \in W,  
     \Phi(\varphi) = q^{M}
    \right\}}. $

The Tauberian theorem for function fields can be derived from the Tauberian theorem for power series in \cite[Thm.~11.2]{Odlyzko1995}.  By \eqref{removesubgroups} and Proposition~\ref{which_cha_contri}\textit{(2)}, the terms with $H\lneq G$ are holomorphic in a neighbourhood of the line $\Real(s)=a(G)$ or have singularities there of order strictly smaller than $b(k,G)$; hence the leading Laurent coefficients of $F_{G,W}(s)$ at the points of $B(k,G)$ coincide with those of $\widetilde F_{G,W}(s)$, and we compute with the latter. Applying the Tauberian theorem to $F_G(s),$ we have
\begin{align*}
N(k,G,W,q^M) \sim \frac{\log (q)^b}{(b-1)!} \left(\sum_{k=0}^{|B(k,G)|-1} e^{2 \pi i k M / |B(k,G)|} r_k \right)M^{b-1}q^{a M}, \\ \text { as } M \rightarrow \infty, M \in \mathbb{Z}, 
\end{align*}
with
$
r_h:= \lim _{s \rightarrow s_h}\left(s-s_h\right)^b F_G(s)
$. Note that $s_h \in B(k,G)$ are evenly spread by Corollary \ref{infoaboutpoles}. 
Note that $\sum_{k=0}^{|B(k,G)|-1} e^{2 \pi i k M / |B(k,G)|} r_h $ could be zero when multiple poles are involved, since cancellation can occur. For example, if we apply the Tauberian theorem to $\zeta_{\F_3(t)(\mu_2)}(s)$, all the terms corresponding to odd exponents will vanish. Therefore, it only makes sense to talk about the asymptotic behavior for those even exponents. One way to take this into account is by observing this cancellation will not happen over all arithmetic progressions modulo $|B(k,G)|$, otherwise, this would contradict the fact that there is a pole located at $s = a(G)$. In this way,  we may deduce that there exists some $N_0$, with
\begin{align*}
    N(k,G,W,q^M) \sim\frac{\log (q)^b}{(b-1)!}\left(\sum_{k=0}^{|B(k,G)|-1} e^{2 \pi i k N_0 / |B(k,G)|} r_h \right)M^{b-1} q^{a M},\\ \text { as } M \rightarrow \infty, \quad M \equiv N_0(\bmod N).
\end{align*}

This method of handling the potential cancellation was used by Lagemann in his paper \cite[Thm.~1.2]{Lagemann_2012} and \cite[Thm.~1.2]{Lagemann_2015}. As there are multiple poles to consider, it is hard to calculate the leading constant explicitly, and Lagemann did not perform this calculation in his later paper \cite[\S1]{Lagemann_2015},  when treating the general purely wild case.

Another method, proposed by Wright in \cite[Thm.~1.3]{10.1112/plms/s3-58.1.17}, is to consider a weighted average of $N(k,G,W,q^M)$:
\begin{equation*}
    N_{av}(k,G,W,q^M):=\frac{1}{|B(k,G)|}\sum_{j = 0}^{|B(k,G)|-1} \frac{N\left(k, G, W,q^{M+j}\right)}{q^{a j}}.
\end{equation*}
By doing this, we can cancel the contribution of poles from $B(k,G)$ that are not located at $a(G)$: 
$$
N_{av}(k,G,W,q^M)\sim\left(\frac{\log (q)^{b(k,G)}}{(b(k,G)-1)!} (r_0)\right)  q^{ a(G)M}M^{b(k,G)-1}, \quad \text { as } M \rightarrow \infty \text { and } M \in\Z,
$$
where $r_0 = \lim_{s \to a}(s-a)^{b(k,G)}F_G(s).$

Now we are ready to calculate the leading constant. We first focus on the case where $G_w$ is not cyclic. By Corollary \ref{infoaboutpoles}, we have 
\begin{equation*}
            \Theta(k,G,W;s) = D(k,G,W;s)\prod^{p^e}_{\substack{j=2\\p^f \mid j}}\prod_{\alpha\,\mid \, |G_t|}\zeta_{k(\mu_\alpha)}\left(js-\sum_{m=1}^e\left(j-1-\left\lfloor\frac{j-1}{p^m}\right\rfloor\right)\right)^{d(k,\alpha)} .
        \end{equation*}
 Therefore, together with Proposition~\ref{PossionS:the prop}, we have that 
\begin{align*}
    \frac{|G|}{|\F_q^{^*}\otimes G|}D(k,G,W;a) &= \lim_{s \to a}F_G(s)\prod^{p^e}_{\substack{j=2\\v_p(j-1)\leq e-1\\p^f\mid j}}\prod_{\alpha\,\mid \, |G_t|}\zeta_{k(\mu_\alpha)}\left(js-\sum_{m=1}^e\left(j-1-\left\lfloor\frac{j-1}{p^m}\right\rfloor\right)\right)^{-d(k,\alpha)} 
    \\&=\lim_{s \to a}(s-a)^bF_G(s) \\ &\times\prod^{p^e}_{\substack{j=2\\v_p(j-1)\leq e-1\\p^f \mid j}}\prod_{\alpha\,\mid \, |G_t|}\frac{\zeta_{k(\mu_\alpha)}\left(js-\sum_{m=1}^e\left(j-1-\left\lfloor\frac{j-1}{p^m}\right\rfloor\right)\right)^{-d(k,\alpha)} }{(s-a)^{d(k,\alpha)}}
    \\&= \lim_{s \to a}(s-a)^bF_G(s) \\ &\times\prod^{p^e}_{\substack{j=2\\v_p(j-1)\leq e-1\\p^f\mid j}}\prod_{\alpha\,\mid \, |G_t|}\left(\operatorname{Res}_{s = a}\zeta_{k(\mu_\alpha)}\left(js-\sum_{m=1}^e\left(j-1-\left\lfloor\frac{j-1}{p^m}\right\rfloor\right)\right)\right)^{-d(k,\alpha)} .
\end{align*}
Hence, 
\begin{align*}
    r_0 &=\frac{|G|}{|\F_q^{^*}\otimes G|} D(k,G,W;a)\prod^{p^e}_{\substack{j=2\\p^f\mid j}}\prod_{\alpha\,\mid \, |G_t|}\left(\operatorname{Res}_{s = a}\zeta_{k(\mu_\alpha)}\left(js-\sum_{m=1}^e\left(j-1-\left\lfloor\frac{j-1}{p^m}\right\rfloor\right)\right)\right)^{d(k,\alpha)}
\\&= \frac{|G|}{|\F_q^{^*}\otimes G|}D(k,G,W;a)\left(\prod_{\substack{1\leq j \leq p^e \\ p^f\mid j }}\frac{1}{j}\right)^{\left( \sum_{\alpha\,\mid \, |G_t|}d(k,\alpha)\right)}
\left(\prod_{\alpha\,\mid \, |G_t|}\left(\operatorname{Res}_{s = 1}\zeta_{k(\mu_\alpha)}\left(s\right)\right)^{d(k,\alpha)}\right)^{p^{e-f}}.
\end{align*}
Here the second equality is due to  the fact that $\operatorname{Res}_{s = a}\zeta_{k(\mu_\alpha)}\left(js-\sum_{m=1}^e\left(j-1-\left\lfloor\frac{j-1}{p^m}\right\rfloor\right)\right)$ $=\frac{1}{j}\operatorname{Res}_{s = 1}\zeta_{k(\mu_\alpha)}\left(s\right)$.
By \cite[Chapter VIII, \S 3]{cassels1967algebraic}, we have $\zeta_{L}(s) = \mathrm{L}(\operatorname{Ind}_{k}^L(\CC),s)$. Therefore, we may rewrite our expression of the leading constant using Artin L-functions in Theorem \ref{maintheorem}.

By Corollary \ref{infoaboutpoles}, $D(k,G,W;s)$ is holomorphic for $\Real(s) \geq a(G)- \epsilon$.  Therefore, we may calculate its value at $s=a(G)$, according to Proposition~\ref{prop_trivial_char_decomp}, after decomposing $\Theta(k,G,W;s)$ into an Euler product, we have
\begin{align*}
D(k,G, W;a) &= \prod_{v \in S}\frac{\lambda_v^{-1}}{|G|\Phi_v(\psi_v)^a}\prod_{v \notin S}\left( \frac{1}{|G|}\sum_{\chi_v \in \operatorname{Hom}(k_v^*,G)} \frac{1}{\Phi(\chi_v)^a}\right)\lambda_v^{-1},
\end{align*}
where $\lambda_v$ is defined in (\ref{GaloisModule}). 

For the cyclic wild case, we follow the same procedure. A similar calculation shows that we may use the formula obtained above by taking $f = 0.$ 
Together with Proposition \ref{PossionS:the prop}, we complete the calculation of the leading constant appearing in Theorem~\ref{maintheorem}. Hence the proof of Theorem~\ref{maintheorem}.

\section{Equidistribution for $BG$ } 
This section provides the geometric background for Theorem~\ref{ED}. We define the Tamagawa measure and prove Theorem~\ref{ED}.

\subsection{The groupoid $BG(k)$ and field extensions}
Let $G$ be a constant abelian group scheme over a global field $k$.  According to \cite[Lemma~2.2]{loughran2025mallesconjecturebrauergroups}, the groupoid $B G(k)$ can be viewed as homomorphisms $\Gamma_k \rightarrow G$, with isomorphisms given by conjugations in $G$.

Let $W_v \subseteq BG[k_v],$recall our definition of $a(G)$ in Theorem~\ref{easyED} and definition of $b(k,G)$ in Theorem~\ref{maintheorem}. We define the local Tamagawa measure modelled on \cite[Def.~8.7]{loughran2025mallesconjecturebrauergroups}
 $$\tau_{\Phi, v}(W_v):= \sum_{\psi_v \in [W_v]} \frac{1}{|\operatorname{Aut}(\varphi_v)|\Phi(\varphi_v)^a} = \frac{1}{|G|}\sum_{\psi_v \in [W_v]} \frac{1}{\Phi(\varphi_v)^a} ,$$
 where the last equality is due to the fact that the automorphisms are given by conjugations since  $G$ is abelian. In the setting of \cite{santens2026leadingconstantmallesconjecture}, the sum over $[W_v]$ was finite, so convergence was clear. For us, the convergence is not immediatte and follows from Proposition~\ref{absconvergence:delta}. Recall our definition of $\mathrm{L}(M_G, s)$ from (\ref{GaloisModule}) and
$
\lambda_v=\mathrm{L}_v(M_G, 1).
$ 
Then we can define the global Tamagawa measure
$$
\tau_\Phi=\mathrm{L}^*(M_G, 1) \prod_v \lambda_v^{-1} \tau_{\Phi, v}.
$$
Here $\mathrm{L}^*(M_G, 1):=\lim _{s \rightarrow 1}(s-1)^{b(k, G)} \mathrm{L}(M_G, s)$, which is non-zero as $b(k, G)$ is the order of pole at $s=1$. Write $\hat{G} := \operatorname{Hom}(G,\mathbb{G}_m)$ and by tensor-hom adjunction formula, we have that $|\hat{G(k)}| = \left|\F_q^* \otimes G^{\wedge}\right| $. Now, we can rewrite our asymptotic formula in Theorem \ref{maintheorem} with the leading constant as 
\begin{equation}\label{finalresultofcounting}
   \frac{1}{|G|}N_{av}(k,G,W,q^M) \sim \left(\prod_{\substack{1\leq j \leq p^e \\ p^f\mid j }}\frac{1}{j}\right)^{\left( \sum_{\alpha\,\mid \, |G_t|}d(k,\alpha)\right)}\frac{\tau_\Phi(\prod_vBG[k_v])\log(q)^b}{(b-1)!  |\hat{G}(k)|}q^{aM}M^{b-1}.
\end{equation}
\subsubsection{Comparison of the result}
In \cite[Conjecture~9.1]{loughran2025mallesconjecturebrauergroups}, the authors applied Peyre's formalism for the leading constant in Manin’s conjecture \cite{Peyre} to obtain a prediction for the leading constant in Malle’s conjecture. They predict that if $G$ is tame, then (\ref{finalresultofcounting}) should be
$$
\frac{(\log q)^{b(k,G)} \cdot \left|\operatorname{Br}_{un} B G / \operatorname{Br} k\right| \cdot \tau_\Phi\left(\prod_vBG[k_v]\right)}{\# \widehat{G}(k)(b(k, G)-1)!},
$$
when counting by conductor. The $(\log q)^{b(G)}$ factor comes from the Tauberian theorem we have used. As for $\operatorname{Br}_{un}(BG)$, see \cite[\S 6]{loughran2025mallesconjecturebrauergroups}. Here, we notice that $\operatorname{Br}_{un}BG/\operatorname{Br}k \cong \Sh_{w}^{1}(k,\hat{G})$, by \cite[Lemma~10.21]{loughran2025mallesconjecturebrauergroups}.
However, by Grunwald-Wang \cite[Thm.~9.1.11]{CohoNF}, when $G$ is purely tame, $\Sh_{w}^{1}(k,\hat{G}) = 0$. Therefore, we observe that our result agrees with their conjecture.

When $G$ has a non-trivial wild part, there is no prediction due to the absence of the effective cone and $\operatorname{Br}_{un}BG$, which requires $G$ to be tame. Therefore $\left(\prod_{\substack{1\leq j \leq p^e \\ p^f\mid j }}\frac{1}{j}\right)^{\left( \sum_{\alpha\,\mid \, |G_t|}d(k,\alpha)\right)}$  is a rational number that suggests the shape of the effective cone constant in the wild case.
\subsection{Proof of Theorem \ref{easyED}} For $ \psi_v \in \Hom (\Gamma_{k_v},G)$ with $v \in S$, let $W = \prod_{v \in S}\{\psi_v\} \times \prod_{v \notin S}\Hom(\Gamma_{k_v},G)$. Applying Theorem \ref{maintheorem} yields
\begin{align*}
        \lim_{M \to \infty}\frac{N_{av}(k,G,W,q^M)}{N_{av}(k,G,q^M)} &= \frac{\prod_{v \in S}\frac{\lambda_v^{-1}}{|G|\Phi_v(\chi_v)^a}\prod_{v \notin \Omega_k}\left( \frac{1}{|G|}\sum_{\chi_v \in \operatorname{Hom}(k_v ^*,G)} \frac{1}{\Phi(\chi_v)^a}\right)\lambda_v^{-1}}{\prod_{v \in \Omega_k}\left(\frac{1}{|G|} \sum_{\chi_v \in \operatorname{Hom}(k_v^*,G)} \frac{1}{\Phi(\chi_v)^a}\right)\lambda_v^{-1}}\\ &=\frac{\tau_\Phi(W)}{\tau_\Phi(\prod_vBG[k_v])}. 
\end{align*}
Here the last equality relies on $\Hom(k_v^*,G)\cong \Hom(\Gamma_{k_v},G)$ via the local Artin map.   \quad \quad \qed
\subsection{Proof of Theorem \ref{ED}}The following proof is adapted from the proof of \cite[Prop.~9.11]{loughran2025mallesconjecturebrauergroups}.
We start with the lower bound. As $W$ and its interior has the same measure, we are free to replace $W$ by its interior, so we may assume that $W$ is open. 
By \cite[Proposition 3.5]{Cesnavicius2015}, the topology on $H^1(k_v,G)$ is discrete, and hence $BG[k_v]$ is discrete by the correspondence given in \cite[Lemma 2.2]{loughran2025mallesconjecturebrauergroups}.

Since the topology here is the product topology, we observe that $W$ can be written as a union of basic open sets of the form $\prod_{v \in S}\left\{\psi_v\right\} \times \prod_{v \notin S} B G\left[k_v\right]$ for some varying finite sets $S \subseteq \Omega_k$ and some $\psi_v. $ We claim that actually $W$ can be written as a disjoint union of basic open sets.

Indeed, pick $(\varphi_v)_v \in W$. As $W$ is open, there is a basic open set $W_{\varphi_v} \subseteq W$ that contains $(\varphi_v)_v \in W$. As the topology on $B G\left[k_v\right]$ is discrete,  the basic open sets are also closed. Now, we repeat the same procedure with the open set $W-W_{\varphi_v},$ which proves the claim.

The total measure of $W$ is finite, and we may write $W$ as the disjoint union of sets of basic open sets. Therefore, for any $\epsilon > 0, $ there exists finitely many  $W_j$ with $j \in J_\epsilon$ such that $\sum_{j \in J_\epsilon}\tau_\Phi(W_j)  \geq \tau_\Phi(W) - \epsilon $.
Therefore,
\begin{align*}
     \liminf_{M \to \infty} \frac{N_{av}(k,G,W,q^M)}{N_{av}(k,G,q^M)} &\geq \sum_{j \in J_\epsilon}\liminf_{M \to \infty} \frac{N_{av}(k,G,W_j,q^M)}{N_{av}(k,G,q^M)} 
     \\& =\sum_{j \in J_\epsilon} \frac{\tau_\Phi(W_j) }{\tau_\Phi(\prod_vBG[k_v])}
 \geq \frac{\tau_\Phi(W) - \epsilon}{\tau_\Phi(\prod_vBG[k_v])},
\end{align*}
where the equality on the second line is due to Theorem \ref{easyED}.
The required lower bound is obtained by taking $\epsilon \to 0.$ As the complement of $W$ is also a continuity set, we have the same upper bound by applying the lower bound to the complement of $W$.\quad \quad\quad\quad\quad\quad\qed 
\bibliographystyle{amsalpha}{}
\bibliography{Bibfile}
\end{document}